\documentclass[11pt]{amsart}

\usepackage{amsmath,amsthm,amsfonts,amssymb}
\usepackage{pstricks,pst-node,pst-tree}
\usepackage{color}
\usepackage{tikz-cd}

\newtheorem{thm}[equation]{Theorem}
\newtheorem{cor}[equation]{Corollary}
\newtheorem{lem}[equation]{Lemma}
\newtheorem{prop}[equation]{Proposition}

\theoremstyle{definition}

\numberwithin{equation}{section}
\newtheorem{conv}[equation]{Conventions}

\definecolor{mjo}{rgb}{0,0,.9}

\newcommand{\cc}{\mathbb{C}}
\newcommand{\z}{\mathbb{Z}}
\newcommand{\F}{\mathbb{F}}

\newcommand{\B}{\mathfrak{b}}
\newcommand{\g}{\mathfrak{g}}

\newcommand{\Vir}{{\rm Vir}}

\newcommand{\Ind}{{\rm Ind} }

\definecolor{mjo}{rgb}{.4,0,.9}

\begin{document}

\title{Modules induced from polynomial subalgebras of the Virasoso algebra}

\author[Matthew Ondrus and Emilie Wiesner]{Matthew Ondrus and Emilie Wiesner}

\address{\noindent Mathematics Department,
Weber State University,
Ogden, UT  84408 USA,  \emph{E-mail address}: \tt{mattondrus@weber.edu}}

\address{\noindent Department of Mathematics, Ithaca College, Williams Hall Ithaca, NY 14850, USA, \ \  \emph{E-mail address}: \tt{ewiesner@ithaca.edu}}
 
% \thanks{$^*$ ??. \newline  \newline
\begin{abstract}
The Virasoro Lie algebra is a one-dimensional central extension of the Witt algebra, which can be realized as the Lie algebra of derivations on the algebra $\cc [t^{\pm}]$ of Laurent polynomials.  Using this fact, we define a natural family of subalgebras of the Virasoro algebra, which we call polynomial subalgebras.  We describe the one-dimensional modules for polynomial subalgebras, and we use this description to study the corresponding induced modules for the Virasoro algebra.  We show that these induced modules are frequently simple and generalize a family of recently discovered simple modules.   Additionally, we explore tensor products involving these induced modules.  This allows us to describe new simple modules and also recover results on recently discovered simple modules for the Virasoro algebra.


\end{abstract}

\date{}

\maketitle

\section{Introduction}
Throughout the paper, let $\z$ represent the integers, $\z_{\geq n} = \{k \in \z \mid k \geq n \}$ for any $n \in \z$, $\cc$ the complex numbers, and $\cc^{\times}= \cc \setminus \{ 0 \}$. 

The Virasoro algebra is defined as $\Vir={\rm span}_\cc \{z, e_j \mid j \in \z \}$
with Lie bracket 
\begin{align*}
[e_j, e_k] &= (k-j)e_{j+k} + \delta_{k,-j} \frac{j^3-j}{12} z; \\
[z, e_j] &= 0.
\end{align*}
The Virasoro algebra can be naturally realized as the unique central extension of the derivations on Laurent polyonomials $\cc[t^{\pm}]$, via the correspondence $e_j \leftrightarrow t^{j+1} \frac{d\ }{dt}$. It also has a natural triangular decomposition: $\Vir= \Vir^+ \oplus \left( \cc z \oplus \cc e_0 \right) \oplus \Vir^-$, where $\Vir^\pm = {\rm span} \{ e_{\pm j} \mid j \in \z_{\ge 1} \}$ and $\cc z \oplus \cc e_0$ acts diagonally on $\Vir$. 

The Virasoro algebra plays an important role in mathematical physics via vertex operator algebras (cf. \cite{LL04}). It also arises in connection with the representation theory of affine Kac-Moody Lie algebras (cf. \cite{K90}), and its representation theory presents an interesting and complex case study. 

Much of the past work on the representation theory of the Virasoro algebra has drawn on the underlying triangular decomposition structure. This includes weight modules. Kac \cite{K80} and Feigin and Fuchs \cite{FF90} described the simple highest weight modules. Mathieu \cite{M92} classified irreducible Virasoro algebra modules with finite-dimensional weight spaces, and Mazorchuk and Zhao \cite{MZ07} extended this to irreducible weight modules with at least one finite-dimensional weight space.  There also exist several families of simple weight modules with all infinite-dimensional weight spaces \cite{CGZ13, CM01, LLZ15}. 

Other work has focused on the action of $\Vir^+$, including the study of variations of Whittaker modules for the Virasoro algebra \cite{FJK12, LGZ11, OW08}. Mazorchuk and Zhao \cite{MZ14} gave a unified description of these modules, describing all simple modules such that certain subalgebras of $\Vir^+$ act locally finitely.

Recently, there has been progress in a new direction.  L\"u and Zhao \cite{LZ14} constructed modules $\Omega (\lambda, b)$ (for $\lambda, b \in \cc$) by twisting certain irreducible modules for the associative algebra $\cc[t^\pm, \frac{d}{dt}]$; they showed that these modules are irreducible if and only if $b \neq 1$. (They also showed these modules are isomorphic to the modules defined in \cite{GLZ13}).  Tan and Zhao \cite{TZ16, TZ13} later determined irreducibility conditions for tensor products of (possibly several) $\Omega(\lambda, b)$ and the modules classified by Mazorchuk and Zhao \cite{MZ14}. Moreover, they constructed isomorphisms between these tensor products and other induced modules for $\Vir$; their results recaptured some known Virasoro modules \cite{MW14} as well as producing several families of new simple modules.

In this paper, we further exploit the connection between the Virasoro algebra and $\cc[t^{\pm}]$ to define a family of ``polynomial" subalgebras of $\Vir$ and associated induced modules. We also transfer these results to modules induced from ``restricted polynomial" subalgebras, via tensor products.  For all the modules constructed, we provide irreducibility conditions, showing that these modules are simple for almost all choices of parameters.  The modules induced from polynomial subalgebras of degree one turn out to be isomorphic to $\Omega (\lambda, b)$. Thus, this paper reproduces the results of \cite{TZ16, TZ13} in a more unified way, and significantly extends them.

The paper organization is as follows: In Section \ref{sec:notationVir}, we define polynomial subalgebras and prove that all $\Vir$-subalgebras of co-dimension one are polynomial subalgebras.  In Section \ref{sec:homsAndOneDimReps}, we characterize one-dimensional modules for polynomial subalgebras, with an eye toward inducing these up to full $\Vir$-modules.  In Section \ref{sec:singleroot}, we establish a number of facts about $\Vir$-modules induced from polynomial subalgebras, where the underlying polynomial has only one distinct root; these results lay the groundwork for general irreducibility results for modules induced from polynomial subalgebras.  In Section \ref{sec:tensorproducts}, we consider tensor products of the modules analyzed in Section \ref{sec:singleroot}; we also allow for tensoring with modules $V$ that have a certain nice action of a subalgebra of $\Vir^+$.  Theorem \ref{thm:tensorsimpleV} of Section \ref{sec:tensorproducts}  gives irreducibility conditions for these tensor products.  Proposition \ref{prop:whenIsomorphic} implies that the modules induced from polynomial subalgebras are distinct from the modules $\Omega ( \lambda, b)$ of \cite{LZ14} when the corresponding polynomial is not linear.  Finally, in Section \ref{sec:tensorapplications}, we prove an isomorphism between the module tensor products of Section \ref{sec:tensorproducts} and newly constucted induced modules.

\section{Polynomial Subalgebras of $\Vir$}\label{sec:notationVir} 
In this section we construct subalgebras of $\Vir$ associated with elements of $\cc[t^{\pm}]$. This construction takes advantage of the connection between the Virasoro algebra and derivations of $\cc[t^{\pm}]$.

 Throughout the paper, we treat the space $\cc[t^{\pm}]$ as a Lie algebra with bracket defined by 
\begin{equation} \label{eqn:basicLie}
[f, g] = t (f g' - g f'),
\end{equation}
where $f, g \in \cc [t^{\pm}]$ and $f'$ and $g'$ represent the standard Laurent polynomial derivative with $t^n \mapsto n t^{n-1}$. This structure arises from identifying derivations on $\cc[t^{\pm}]$ with elements of $\cc[t^{\pm}]$ via the linear map $t^{j+1} \frac{d \ }{dt} \rightarrow t^j$; alternatively, we can view this relation as imposing a Lie algebra structure on $\cc[t^{\pm}]$ so that the linear map $\theta: \Vir \rightarrow \cc [t^{\pm}]$ given by $\theta(e_j)=t^j$ and $\theta(z) =0$ is a surjective Lie algebra homomorphism.  

We frequently make use of the associative algebra structure of $\cc[t^\pm]$. To distinguish multiplication operations in $\cc[t^\pm]$ and $U(\cc[t^\pm])$, we establish the following notational conventions. For $f, g \in \cc[t^\pm]$, we write $fg$ to denote the product in the associative algebra $\cc [t^{\pm}]$, and we write $f \cdot g$ to denote the product in the universal enveloping algebra $U( \cc [t^{\pm}])$.  Similarly, for $i \ge 0$ and $g \in\cc [t^{\pm}]$, we write $f^i$ to denote the $i$th power of $f$ in the associative algebra $\cc [t^{\pm}]$, and we write $f^{[i]}$ to denote the $i$th power of $f$ in the universal enveloping algebra $U(\cc [t^{\pm}] )$.  In situations where the multiplicative identities of $\cc[t^{\pm}]$ (as an associative algebra) and $U(\cc[t^{\pm}])$ may be confused, we represent the multiplicative identity of $\cc[t^\pm]$ by $t^0$ and of $U( \cc [t^{\pm}])$ by $1$. (When no opportunity for confusion exists, we continue to write the multiplicative identity of $\cc[t^\pm]$ as $1$ for notational ease.)

%and therefore 
%\begin{equation} \label{eqn:BrackSamePowers}
%[ t^k g^m , t^\ell g^m ] = ( \ell - k) t^{k+ \ell} g^{2m}.
%\end{equation}
Let $\langle f \rangle={\rm span}_\cc \{ t^j f \mid j \in \z \}$, the associative algebra ideal of $\cc [t^{\pm}]$ generated by $f$.  Since $\langle f \rangle = \langle t^j f \rangle=\langle c f \rangle$ for any $j \in \z$ and $c \in \cc^\times$, we may assume that $f \in \cc[t]$; and $f$ has a lead coefficient of $1$ and nonzero constant term.
%: $f= \sum_{i=0}^p a_i t^j$ for $a_p=1$ and some $ a_0 \neq 0$. 
It follows from (\ref{eqn:basicLie}) that 
$[t^j f, t^k f] = (k - j) t^{j+k} f^2 \in \langle f \rangle$, and thus $\langle f \rangle$ is a Lie subalgebra of $\cc [t^{\pm}]$.  Therefore, 
$$
\Vir^{f} = \{ x \in \Vir \mid \theta(x) \in \langle f\rangle \}  = \theta^{-1} ( \langle f \rangle )
$$
is a subalgebra of $\Vir$. In particular, $\Vir^{f}$ has a basis
$$
\{ z, x_j^f \mid j \in \z\}
$$
where  $$x_j^f = a_0 e_j + a_1 e_{j+1} + \cdots + a_p e_{j+p}.$$  
We call any subalgebra of $\Vir$ of the form $\Vir^{f}$ a {\it polynomial subalgebra} of $\Vir$.

%$$[x_k^f, g_\ell^f ] = ( \ell - k) \left( \sum_{i=0}^p a_i g_{k + \ell + i}^f \right) + c z,$$

\subsection{Subalgebras of codimension 1}

Here we show that all $\Vir$ subalgebras of codimension $1$ are polynomial subalgebras $\Vir^f$ for some first degree polynomial $f$.

\begin{lem}\label{lem:codim1noMonomial}
Let $\mathfrak S \subseteq \Vir$ be a Lie subalgebra of codimension $1$.  Then $e_j \not\in \mathfrak S$ for all $j \in \z$.
\end{lem}
\begin{proof}
We first show that $e_0 \not\in \mathfrak S$.  If $e_0 \in \mathfrak S$, this implies that 
$$\mathfrak S = \bigoplus_{k \in \z} \mathfrak S_k,$$
where $\mathfrak S_k = \mathfrak S \cap \cc e_k$ for $k \neq 0$ and $\mathfrak S_0 = \mathfrak S \cap ( \cc e_0 + \cc z)$.  Since $\mathfrak S$ has finite codimension, there can only be finitely many $k \in \z$ such that $\mathfrak S_k = 0$.  As $\mathfrak S_k = \cc e_k$ whenever $\mathfrak S_k \neq 0$ (and $k \neq 0$), it follows that the nonzero $\mathfrak S_k$ generate all of $\Vir$.  Thus $\mathfrak S=\Vir$, a contradiction. Therefore, $e_0 \not\in \mathfrak S$. 

Now suppose that $e_j \in \mathfrak S$ for some $j \neq 0$.  Let $k \in \z$ with $k \not\in \{ 0, j, -j \}$.  Since $\mathfrak S$ has codimension $1$ and $e_0 \not\in \mathfrak S$, there exist $c, d \in \cc$, $d \neq 0$, such that $c e_0 + d e_k \in \mathfrak S$.  Then $\mathfrak S$ contains $[ce_0 + d e_k, e_j] = j c e_j + (j-k)d e_{j+k}$.  As $e_j \in \mathfrak S$ and $(k-j)d \neq 0$, this implies $e_{j+k} \in \mathfrak S$.  This holds for all $k \not\in \{ 0, j, -j \}$. Since $\{ e_{j+k} \mid k \not\in \{ 0, j, -j \} \}$, generates $\Vir$, this forces $\mathfrak S = \Vir$, a contradiction.  Thus $e_j \not\in \mathfrak S$ for all $m \in \z$. 
\end{proof}

The following result shows that any subalgebra $\mathfrak S$ of $\Vir$ of codimension one has the form $\mathfrak S = \Vir^f$ for some $f=t+c$, $c \in \cc^\times$.
\begin{prop} \label{prop:codim1}
Let $\mathfrak S \subseteq \Vir$ be a subspace of codimension $1$.  Then $\mathfrak S$ is a subalgebra if and only if there exists $c \in \cc^\times$ such that $\mathfrak S = {\rm span}_\cc \{ z, e_j + c e_{j + 1} \mid j \in \z \}$. 
\end{prop}
\begin{proof}
Note that  $\mathfrak S={\rm span}_\cc \{ z, e_j + ce_{j+ 1} \mid j \in \z \}$ is a subalgebra of $\Vir$ of codimension $1$ for any $c \in \cc^\times$.  

Now let $\mathfrak S$ be a subalgebra of codimension $1$. Then Lemma \ref{lem:codim1noMonomial} implies that for each $j \in \z$ there exists a unique $c_j \in \cc^\times$ with $e_j + c_j e_{j+1} \in \mathfrak S$.  Therefore, to prove the result, it's enough to show that $e_j + c_0 e_{j+1} \in \mathfrak S$ (forcing $c_j=c_0$) for each $j \in \z$; and $z \in \mathfrak S$.

For each $j \in \z$, write $y_j= e_j + c_j e_{j+1}$.  Throughout the proof, we use the following identity: for any $0 \neq j \in \z$, 
\begin{equation}\label{eqn:abcdLinBrack2}
[y_0, y_j] = (j-1) \left(  e_{j} + c_0 e_{j+1} \right) +  y_j  +  j  c_j  \left(  e_{j+1} + c_0 e_{j+2} \right).  
\end{equation}

We first show $e_j+ c_0 e_{j+1} \in \mathfrak S$ for $j \ge 2$ by induction.  To see that $e_2 + c_0 e_3 \in \mathfrak S$, notice that $[y_0, y_1] = y_1 + c_1 (e_2 + c_0 e_3)$.  Since $y_1, [y_0, y_1] \in \mathfrak S$, this implies that $e_2 + c_0e_3 \in \mathfrak S$.   
Now assume the result holds for some $j \geq 2$. Then $y_j=e_j+c_0 e_{j+1} \in \mathfrak S$, and (\ref{eqn:abcdLinBrack2}) becomes 
$[y_0 , y_j ] = jy_j + jc_j (e_{j+1} + c_0 e_{j+2}).$
This forces $e_{j+1} + c_0 e_{j+2}\in \mathfrak S$. By induction, $e_j + c_0 e_{j+1}\in \mathfrak S$ for all $j \in \z_{\ge 2}$.

%$y_{n+1} = e_{n+1} + c_0 e_{n+2}$. 

\medskip

We also use induction to argue $e_{-j} +c_0 e_{-j+1} \in \mathfrak S$ for $j >0$.   In the base case $j=1$, (\ref{eqn:abcdLinBrack2}) becomes
$[y_0, y_{-1} ] = -2(e_{-1}+c_0e_0)+y_{-1}- c_{-1}  y_0.$
It follows that $e_{-1}+c_0e_0 \in \mathfrak S$, so that $y_{-1} = e_{-1}+c_0e_0$.  Now assume $j >1$ and suppose $e_{-k}+c_0 e_{-k+1} \in \mathfrak S$ for $k < j$. This forces $y_{-k} = e_{-k} + c_0 e_{-k+1}$ for $k<j$, and (\ref{eqn:abcdLinBrack2}) becomes 
$$
[y_0, y_{-j}] = -(j+1) (e_{-j}+c_0e_{-j+1}) + y_{-j} - j c_{-j} y_{-(j-1)}.
$$
Consequently $e_{-j}+c_0e_{-j+1} \in \mathfrak S$ for all $j \in \z_{\ge 1}$.  

We have now shown that $e_j+ c_0 e_{j + 1} \in \mathfrak S$ for all $j \neq 1$.  For $j=1$, note that
$
[y_{-1}, y_2]=3(e_1+c_0e_2) + 2c_0 y_2.
$
Therefore, $e_1+c_0e_2\in \mathfrak S$. Lastly, we verify that $z \in \mathfrak S$ by the computation
$$
[y_{-2}, y_2]=y_0+4c_0y_1-\frac12 z.
$$
This completes the proof.
\end{proof}

\section{One-dimensional representations of $\Vir^{f}$}\label{sec:homsAndOneDimReps}

The central objects of study in this paper are $\Vir$-modules that are induced from one-dimensional modules for polynomial subalgebras. With this in mind, we use this section to consider Lie algebra homomorphisms $\mu: \Vir^f \rightarrow \cc$. Recall that we may assume that $f \in \cc[t]$ with lead coefficient $1$ and nonzero constant term. 

\begin{lem} \label{lem:homzerocentral}
Suppose $f= \sum_{i=0}^p a_i t^j $, where $p \geq 1$, $a_p=1$, and $ a_0 \neq 0$. If $\mu : \Vir^{f} \to \cc$ is a homomorphism, then $\mu (z) = 0$. 
\end{lem}
\begin{proof} 

Recall the homomorphism $\theta: \Vir\rightarrow \cc[t^{\pm}]$. Then for $j,k \in \z$, (\ref{eqn:basicLie}) implies 
$\theta ( [x_j^f, x_k^f ] ) = [ t^j f, t^k f] = ( k-j ) t^{j+k} f^2.$
Therefore,
$$
[x_j^f, x_k^f ] = ( k-j) \left( \sum_{i=0}^p a_i x_{j + k + i}^f \right) + c_{j, k} z
$$
for some $c_{j,k} \in \cc$. For $p \ge 0$, we have 
\begin{align*}
[ x_{-(p+1)}^f, x_{(p+1)}^f ] = 2(p+1) \sum_{i=0}^p a_i x_{i}^f + a_0^2 \left( \frac{(p+1)-(p+1)^3}{12} \right) z,
\end{align*}
and in particular 
$$[x_{-1}^f, x_1^f ] =  2 \sum_{i=0}^p a_i x_{i}^f.$$
 Because $\mu ( [ \Vir^{f}, \Vir^f] ) = 0$, we have
$$
0 = \mu([ x_{-(p+1)}^f, x_{(p+1)}^f ]- (p+1) [x_{-1}^f, x_1^f ]) = a_0^2 \left( \frac{(p+1)-(p+1)^3}{12} \right) \mu (z). 
$$ 
Since $a_0 \neq 0$ and $p+1 \ge 2$, it follows that $\mu (z) = 0$.
\end{proof}

This lemma implies that any homomorphism $\mu : \Vir^{f} \to \cc$ defines a homomorphism $\mu' : \langle f \rangle \to \cc$ (which we identify with $\mu$), where $\mu'(t^k f)= \mu(x_k^f)$. Since any map $\mu' : \langle f \rangle \to \cc$ extends to a map $\mu = \mu' \circ \theta : \Vir^f \to \cc$, we may treat $V^{f}_{\mu}$ interchangeably as a $\Vir$-module or a $\cc[t^{\pm}]$-module, where the $\cc[t^{\pm}]$-action on $V^{f}_{\mu}$ is given by $t^k v= e_k v$  for $k \in \z$, $v \in V^{f}_{\mu}$. 
%That is, the following diagram commutes:

%\begin{equation} \label{eqn:CtVirdiagram}
%\begin{tikzcd}
%\Vir \arrow{r}{\theta} \arrow[swap]{d} & \cc[t^{\pm}] \arrow{ld}\\
%{\rm End}(V^{f}_{\mu})  & 
%\end{tikzcd}
%\end{equation}

The following lemma characterizes all homomorphisms $\mu: \langle f \rangle \rightarrow \cc$.
\begin{lem} \label{lem:mupoly}
Let $\lambda_1, \ldots, \lambda_k \in \cc^\times$ be distinct, and $n_1, \ldots, n_k \in \z_{\geq 1}$. Define $f=\prod_{i=1}^k (t-\lambda_i)^{n_i} $.   Then a linear map $\mu: \langle f \rangle \rightarrow \cc$ is a homomorphism if and only if $\mu(t^j f) =p_1(j)\lambda_1^j+ \cdots + p_k(j) \lambda_k^j$ for some polynomials $p_1, \ldots,p_k$ such that $\deg(p_i) < n_i$.
\end{lem}
\begin{proof}
Consider a linear map $\mu: \langle f \rangle \rightarrow \cc$. For each $j \in \z$, write $\mu_j = \mu ( t^j f)$.  Then $\mu$ is a well-defined homomorphism if and only if 
\begin{equation} \label{eqn:recurrence}
0= \mu ( [t^j f , t^\ell f ]) =(\ell-j) \sum_{i=0}^p a_i \mu(t^{j+\ell+i}f) = (\ell-j) \sum_{i=0}^p a_i \mu_{j+\ell+ i}
\end{equation}
for all $j,\ell \in \z$.  This implies that $\sum_{i=0}^p a_i \mu_{m+i} = 0$ for all $m \in \z$ (e.g. take $j = 0$ and $\ell \neq 0$, then take $\ell = 1$ and $j= -1$).  Thus the sequence $\mu_j$ form a linear homogeneous recurrence relation with characteristic polynomial $f$.  It it follows from the theory of recurrence relations that the solutions to this relation $\mu_j$, $j \geq 0$, are precisely of the form $\mu_j =p_1(j)\lambda_1^j+ \cdots + p_k(j) \lambda_k^j$ for polynomials $p_1, \ldots, p_k$ with $\deg(p_i) < n_i$. (See, for example, \cite[Cor. 2.24]{Elaydi95}.)  Moreover, by varying the index of the initial term (from $0$ to any integer), we obtain this solution for all $j \in \z$.  %{\color{red}I like this.  I'll keep thinking about it, but it seems good at first glance.}
\end{proof}

Define the degree of the zero polynomial to be $-1$. 

%{\color{blue}Note: In the following lemma, the formula $\mu ( t^j f^n) = p(i) \lambda^{i+n}$ (as opposed to $\mu ( t^j f^n) = p(i) \lambda^i$) felt potentially confusing, particularly in light of the notation in our previous lemma.  Hence the change \ldots Of course, there is a trade-off with the extra $\lambda$ that appears now. }

\begin{lem} \label{lem:degreehom}
Let $\lambda \in \cc^\times$, $n \in \z_{\ge 1}$, and define $f=t-\lambda$. Suppose $\mu: \langle f^n\rangle  \rightarrow \cc$ is a homomorphism such that $\mu(t^j f^n) = p(j) \lambda^j$ for some nonzero polynomial $p(t) \in \cc[t]$ of degree $r < n$.  Then 
\begin{itemize}
\item there are finitely many pairs $(j, m)$ with $j \in \z$ and $n \leq m \leq n+r$ such that $\mu(t^j f^m) =0$;
\item $\mu(t^j f^m)=0$ for all pairs $(j, m)$ such that $j \in \z$ and $m >n+r$.
\end{itemize}
\end{lem}
\begin{proof}
To prove this lemma, it is enough to show that $\mu(t^j f^m) = p_m(j) \lambda^j$ for a polynomial $p_m$ where 
\begin{itemize}
\item $\deg (p_m)=n+r -m$  if $m \leq n+r$;
\item  $p_m=0$  if $m > n+r$.
\end{itemize}
We prove this by induction on $m$.  For $m=n$, this follows from the assumptions of the lemma.  For $m>n$, we use the inductive hypothesis to conclude
\begin{align*}
\mu(t^j f^m)&= \mu(t^{j+1} f^{m-1} -\lambda t^j f^{m-1}) \\
&= \mu(t^{j+1} f^{m-1}) -\lambda \mu(t^j f^{m-1})\\
&= \left( p_{m-1}(j+1) - p_{m-1}(j)  \right) \lambda^{j+1}.
\end{align*}
If we let $p_m(x) = \lambda (p_{m-1}(x+1) - p_{m-1}(x))$, it follows from the binomial theorem that if $p_{m-1} \neq 0$, then $\deg ( p_m) = \deg (p_{m-1}) -1$.  
\end{proof}

%\begin{defn}
Using the notation of the Lemma \ref{lem:degreehom}, suppose $\mu: \Vir^{(t-\lambda)^n} \rightarrow \cc$ is a homomorphism such that $\mu ( t^j f^n) = p(j) \lambda^j$ for some polynomial $p$ of degree $r$.  Then we say that $p$ is {\it associated} to $\mu$, and $\mu$ has {\it degree} $r$.
%\end{defn}

\section{The induced modules $V^{(t-\lambda)^n}_{\mu}$} \label{sec:singleroot}
The modules we wish to consider in this paper are 
$$
V^{f}_{\mu} = U(\Vir) \otimes_{\Vir^{(t-\lambda)^n}} \cc_{\mu}
$$
for arbitrary polynomials $f \in \cc[t]$ and Lie algebra homomorphisms $\mu: \Vir^f \rightarrow \cc$. As a first step in understanding these modules, we restrict to polynomials $f=(t-\lambda)^n$ where $n \in \z_{\ge 1}$ and $\lambda \in \cc^\times$.   
Section \ref{subsec:basisNotation} establishes bases for these modules. In Section \ref{sec:n=1}, we further restrict to the case $f=t-\lambda$; we believe the simplicity of the calculations in this case serve to highlight the structure of the modules $V_\mu^f$. The remainder of the section is used to establish a number of technical lemmas about $V^{(t-\lambda)^n}_\mu$,  which we use to study (in the Section \ref{sec:tensorproducts}) $V^{f}_{\mu}$ for general polynomials $f$.

\begin{conv} \label{conv}
We use the following notation throughout this section.  Fix $\lambda \in \cc^\times$, $n \in \z_{\ge 1}$, and define $f=t-\lambda$. Let $\mu: \Vir^{f^n} \rightarrow \cc$ be a homomorphism, $p$ the polynomial associated to $\mu$, and write $\deg(p)=r$. We summarize this as $(\lambda, n, f; \mu, p, r)$. Finally, write $v_\mu=1 \otimes 1 \in V_{\mu}^{f^n}$, the canonical generator of $V_{\mu}^{f^n}$.  
 \end{conv}
Also throughout this section, we implicitly use Lemma \ref{lem:homzerocentral}  to treat $V^{f^n}_{\mu}$ at a $\cc [t^{\pm}]$-module. We also make frequent use of the calculation
$$
[t^j f^m, t^k f^\ell] = (\ell - m)t^{j+k+1} f^{\ell+m-1} +(k-j) t^{j+k} f^{\ell+m}. 
$$

\subsection{Bases for $V^{(t-\lambda)^n}_{\mu}$ and associated notation}\label{subsec:basisNotation}

 In the following lemma, we use the notation $u^{[s]}$ from Section \ref{sec:notationVir}, where $u \in \cc[t^{\pm}]$ and $s \in \z_{\geq 0}$.

\begin{lem} \label{lem:inducedbases}
The set 
\begin{align}
\{ (f^0)^{[s_0]} \cdot (f^1)^{[s_1]} \cdots (f^{n-1})^{[s_{n-1}]} v_\mu \mid s_0, \ldots, s_{n-1} \in \z_{\ge 0} \}; \label{basis1} 
\end{align}
is a basis for the module $V^{f^n}_{\mu}$.  %{\color{red}Question: Technically, since we are using $\cdot$ for multiplication in the enveloping algebra, we should use $\cdot \cdots \cdot$ for the missing terms.  But that is pretty dotty.  What do you think?} {\color{magenta} Although technically correct, I find the extra dots distracting and borderline confusing. I would prefer to either leave out the dots or to change notation for the product in the universal enveloping algebra. Maybe $\circ$ or $\bullet$ if we wanted different notation.}
%\begin{align}
%&
%\{ (f^0)^{[s_0]} \cdot (f^1)^{[s_1]} \cdots (f^{n-1})^{[s_{n-1}]} v_\mu \mid s_0, \ldots, s_{n-1} \in \z_{\ge 0} \}; \label{basis1} %\\
%&\{ (t^0)^{[s_0]} \cdot (t^1)^{[s_1]} \cdots (t^{n-1})^{[s_{n-1}]} v_\mu \mid s_0, \ldots, s_{n-1} \in \z_{\ge 0} \}. \label{basis2}
%\end{align}
\end{lem}
%{\color{green} I think paper is now such that we never use the second basis. If that's the case, we can eliminate it from this lemma.}
\begin{proof}
Note that the basis $\{ t^j f^n \mid i \in \z\}$ for $\langle f^n \rangle$ can be extended to a basis for $\cc[t^{\pm}]$ by the set 
$$\{ f^0, f^1, \ldots, f^{n-1} \}.$$
Therefore, from \cite[PBW Theorem, Cor. 2]{MP95}, $U(\cc[t^{\pm}])$ is a free $\langle f^n \rangle$-module with basis 
\begin{align*}
%&\{ e_0^{k_0} e_1^{s_1} \cdots e_{n-1}^{s_{n-1}} \mid k_0, \ldots, s_{n-1} \in \z_{\geq 0} \}\\
&\{ (f^0)^{[s_0]} \cdot (f^1)^{[s_1]} \cdots (f^{n-1})^{[s_{n-1}]} \mid s_0, \ldots, s_{n-1} \in \z_{\ge 0} \}.
\end{align*}  It follows that (\ref{basis1}) is a basis.  %A similar argument can be applied to (\ref{basis2}).
\end{proof}

We now establish notation to describe this basis.  Let $\z_{\ge 0}^n$ denote the set of all $n$-tuples of elements of $\z_{\ge 0}$.  For ${\overline s} = (s_0, \ldots, s_{n-1}) \in \z_{\ge 0}^n$, define 
\begin{align*}
f^{\overline s} &= (f^0)^{[s_0]}  \cdot (f^1)^{[s_1]}  \cdot \cdots \cdot (f^{n-1})^{[s_{n-1}]}. %\\
%t^{\overline s} &= (t^0)^{[s_0]}  \cdot (t^1)^{[s_1]}  \cdot \cdots \cdot (t^{n-1})^{[s_{n-1}]}, 
\end{align*}
Then $\{ f^{\bar s} v_\mu \mid \bar s \in \z_{\ge 0}^n \}$ is a basis for $V^{f^n}_\mu$. Define
$$| \overline s| = \sum_{i=0}^{n-1} s_i.$$
  Define $\ell : \z_{\ge 0}^n \setminus \{ (0, \ldots , 0) \} \to \{ 0, \ldots , n-1 \}$ by 
$$\ell ( \overline s) = \min \{ i \mid s_i > 0 \}$$
and 
\begin{align*}
{\mathcal D}({\overline s})&= (0, \ldots, 0, s_{\ell(\bar s)} - 1, s_{{\ell(\bar s)}+1}, \ldots, s_{n-1} ); \\
{\mathcal D}^j({\overline s})
&=(0, \ldots, 0, s_{\ell(\bar s)} - j, s_{{\ell(\bar s)}+1}, \ldots, s_{n-1} ) \quad \mbox{for $1 \leq j \leq s_{\ell (\overline s)}$} \\&=\underbrace{\mathcal D \circ \cdots \circ \mathcal D}_{\mbox{$j$ times}}.
\end{align*}  
Note that $f^{\overline s} = f^{\ell(\bar s)} \cdot f^{{\mathcal D}({\overline s})}$.

We use a lexicographic ordering on $\z_{\ge 0}^n$: for ${\overline s } = (s_0, \ldots , s_{n-1}), \overline r=(r_0, \ldots, r_{n-1})\in \z_{\ge 0}^n$, write ${\overline r} \succ {\overline s }$ if there exists $j \in \{ 0, \ldots, n-1 \}$ such that $r_i = s_i$ for $i < j$ and $r_{j} > s_j$.

\subsection{$V^f_\mu$ for linear polynomials} \label{sec:n=1}

%{\color{magenta} We can reproduce all results from the TZ papers, somewhat simplifying the arguments.}

 In this section, we consider the induced module $V^f_\mu$ where $f=t-\lambda$ ($\lambda \in \cc^\times$) is a linear polynomial, giving irreducibility conditions on $\mu$ and showing that these modules are isomorphic to the modules $\Omega(\lambda, b)$ constructed in \cite{LZ14}. The results presented in this section are fully contained in the more general results presented later in the paper. However, we believe that considering this special case helps to illuminate the benefits of the induced construction of the modules $V_\mu^f$.

For $k > 0$, it is easy to verify that (in the associative algebra $\cc [t, t^{-1}]$)
\begin{align*}
t^k - \lambda^k t^0 &= \left( \sum_{i=0}^{k-1} t^{k-1-i} \lambda^i  \right)f.
%\\t^{-k} - \lambda^{-k} t^0 &= - (t^{-k} \lambda^{-k}) ( t^k - \lambda^k) = - \left( \sum_{i=0}^{k-1} t^{-1-i} \lambda^{i-k} \right) f.
\end{align*}
A similar calculation follows for $-k$, using $t^{-k} - \lambda^{-k} t^0 = - (t^{-k} \lambda^{-k}) ( t^k - \lambda^k)$.  Applying Lemma \ref{lem:mupoly}, for all $k \in \z$, we have
\begin{equation}\label{eqn:t^sv_mu}
t^k v_\mu = \lambda^k t^0 v_\mu + k \lambda^{k-1} \mu_0 v_\mu, \quad \mbox{where} \ \mu_0=\mu(f).
\end{equation}

\begin{prop}\label{thm:t-1_simple}
Let $f = t - 1$ and $\mu: \Vir^f \rightarrow \cc$ be a Lie algebra homomorphism. Then $V_\mu^f$ is simple if and only if $\mu_0 = \mu (f) \neq 0$. 
\end{prop}
\begin{proof}

For any $h(t) =\sum_i a_i t^i\in \cc[t]$, define $h(t^0)= \sum_i c_i (t^0)^{[i]}v_\mu$.  Then Lemma \ref{lem:inducedbases} implies that for any $ v \in V_\mu^f$, there is some $h \in \cc[t]$ such that $v=h(t^0) v_\mu$. 

Suppose $\mu_0 \neq 0$, and let $M$ be a nontrivial submodule of $V_\mu^f$. We argue that $M= V_\mu^f$. Choose $0 \neq v \in M$ of the form $v = h(t^0) v_\mu$ with $\deg (h)$ minimal.  If $\deg (h) = 0$, then clearly $M = V_\mu$. Now we suppose $\deg (h) > 0$ and produce a contradiction. In particular, it is enough to show that there is some polynomial $0 \neq g \in \cc[t]$ such that $\deg(g) \leq \deg(h)$, $g \not\in \cc h$, and $g(t^0) v_{\mu} \in M$.

Write $h(t) = \sum_i c_i t$.  Note that $[t^k, t^0] = - k t^k$ implies $t^k \cdot (t^0)^{[i]} = ( t^0 - k)^{[i]} \cdot t^k$. Then it follows from (\ref{eqn:t^sv_mu}) that
\begin{align*}
(t^k-t^0) \cdot v &=(t^k-t^0) \cdot h(t^0) v_{\mu}\\
&=\sum_{i=0}^n c_i ( t^0 - k)^{[i]} \cdot t^k v_\mu - \sum_{i=0}^n c_i (t^0)^{[i+1]}v_\mu\\
& = \sum_{i=0}^n c_i ( t^0 - k)^{[i]} \cdot (  t^0 + k \mu_0  ) v_\mu- \sum_{i=0}^n c_i (t^0)^{[i+1]}v_\mu.
\end{align*}
That is, $(t^k - t^0) \cdot v = g_k (t^0) v \in M$, where $g_k (t)= h(t -k) (t + k \mu_0) - h(t) t$.
It is clear that $\deg (g_k) \le \deg(h)$.  

To complete the proof, we show that there exists $k \in \z$ such that $g_k \not\in \cc h$.  Let $\xi \in \cc$ such that $h( \xi ) = 0$.  On the other hand,
$$g_k ( \xi ) = h(\xi -k) ( \xi + k \mu_0) - h( \xi ) \xi = h(\xi -k) ( \xi + k \mu_0).$$
Since $\mu_0 \neq 0$, then $\xi + k \mu_0 \neq 0$ for $k \neq -\xi/\mu_0$.  Since $\deg(h)>0$, it follows that $g_k ( \xi )$ is a nonzero polynomial in $k$, and so there are only finitely many $k$ such that $g_k(\xi)= 0$.  Since $g_k \not\in \cc h$ for $k$ such that $g_k(\xi)\neq 0$, this proves that $M=V^f_\mu$. Thus $V_\mu^f$ is simple.

If $\mu_0=0$, then it follows from (\ref{eqn:t^sv_mu}) and $t^k \cdot (t^0)^{[i]} = ( t^0 - k)^{[i]} \cdot t^k$, that $M= {\rm span} \{(t^0)^{[i]} v_\mu \mid i \in \z_{\geq 1}\}$ is a proper submodule of $V_\mu^f$.
\end{proof}

In the previous proof, the hypothesis that $f=t-1$, rather than a more general linear polynomial $f=t-\lambda$ ($\lambda \in \cc^\times$), allowed us to treat $g_k$ as a polynomial in $k$. To generalize this result to any linear polynomial $f=t-\lambda$, we twist $V_\mu^f$ by a $\Vir$-automorphism.
For a $\Vir$-module $V$ and a Lie algebra automorphism $\tau : \Vir \to \Vir$, define the $\Vir$-module $V^\tau$ where $V^{\tau}=V$ as a vector space and with $\Vir$-action given by $x.v = \tau (x) v$ for $x \in \Vir$ and $v \in V$.  Clearly $V$ is simple if and only if $V^\tau$ is simple.

\begin{cor}
Let $ \lambda \in \cc^\times$, $f=t-\lambda$, and $\mu : \Vir^f \to \cc$.  Then $V_\mu^f$ is simple if and only if $\mu(f) \neq 0$.  
\end{cor}
\begin{proof}
Define a $\Vir$-automorphism $\tau_\lambda : \Vir \to \Vir$ by $\tau_\lambda ( e_k) = \lambda^k e_k$ ($k \in \z$) and $\tau_\lambda (z) = z$.  It is straightforward to see that 
$V_\mu^f \cong (V_{\mu \circ \tau_{\lambda}^{-1}}^{t-1})^{\tau_{\lambda}} $. The result follows from Proposition \ref{thm:t-1_simple}.
\end{proof}

For $\lambda \in \cc^\times$ and $b \in \cc$, \cite{LZ14} define  a $\Vir$-module $\Omega(\lambda, b)$ with basis $\{\partial^k \mid k \in \z_{\geq 0} \}$ and with $\Vir$-action
\begin{equation} \label{eqn:Omegaaction}
e_k \partial^n=\lambda^k (\partial + k(b-1))(\partial -k)^n, \quad k \in \z; \quad z \partial^n=0.
\end{equation}
They show that $\Omega (\lambda, b)$ is irreducible for $b \neq 1$.  Below, we show that these modules are isomorphic to $V_\mu^{t - \lambda}$. (\cite{TZ16} also show that $\Omega(\lambda, b)$ are isomorphic to modules introduced in \cite{GLZ13}.) %See Lemma 13 of \cite{TZ16}. 

\begin{prop}
Let $\lambda \in \cc^\times$, $b \in \cc$.  Then $\Omega (\lambda, b) \cong V_\mu^{t - \lambda}$, where $\mu$ is defined by $\mu(e_{k+1} - \lambda e_k)=\lambda^{k+1} (b-1)$ and $\mu(z)=0$. 
\end{prop}
\begin{proof}
Let $x_k = e_{k+1} - \lambda e_k$. In the module $\Omega ( \lambda, b)$, 
$$x_k .1= \lambda^{k+1} (\partial + (k+1)  (b-1)) - \lambda^{k+1} (\partial + k (b-1)) = \lambda^{k+1} (b-1).$$
Thus, there is a $\Vir$-module homomorphism $\theta : V_{\lambda, \mu} \to \Omega ( \lambda, \frac{\mu}{\lambda} + 1)$ with $\theta ( v_\mu ) = 1$.  Since $\Omega ( \lambda, b)$ is obviously generated by the vector $1$, the map is surjective. Moreover, the standard basis for $V_{\lambda, \mu}$ is sent to the standard basis for $\Omega ( \lambda, \frac{\mu}{\lambda} + 1)$, so that the map is injective.
\end{proof}

\subsection{Computational facts about $V^{f^n}_{\mu}$}
Here we establish several computational facts about $V^{f^n}_{\mu}$ that will be used in the following section.

\begin{lem}\label{lem:repRootPowerComp1}
Let $(\lambda, n, f; \mu, p, r)$ be as in Conventions \ref{conv}. Choose ${\overline s} = (s_0, \ldots, s_{n-1}) \in \z_{\ge 0}^n$ such that $| \overline s|>0$ and ${\ell(\bar s)} >0$; and $i, m \in \z$, such that $m \ge n$ and $m \ge n+r+1 - {\ell(\bar s)}$. Then in $V_\mu^{f^n}$, 
$$[t^j f^m, f^{\overline s} ] v_\mu = \left\{ \begin{array}{ll} 0 & \mbox{if $m > n+r+1 - {\ell(\bar s)}$} \\
({\ell(\bar s)}- m ) s_{\ell(\bar s)} \mu ( t^{j+1} f^{n+r}) f^{{\mathcal D}({\overline s})} v_\mu & \mbox{if $m = n + r + 1 - {\ell(\bar s)}$} \\
 & \mbox{and ${\ell(\bar s)}> 1$}  \\
\sum_{j=1}^{s_1}  (1 - m)^j {s_1 \choose j} \mu (t^{i+j} f^m) f^{{\mathcal D}^j({\overline s })}  v_\mu & \mbox{if $m=n+r$} \\
& \mbox{ and ${\ell(\bar s)}= 1$} 
\end{array} \right.$$
\end{lem}

\begin{proof}
We use the following throughout the proof. Write $\ell={\ell(\bar s)}$. Also, since $m \ge n$, we have that $t^j f^m \in \langle f^n \rangle$ for all $j \in \z$ and thus $(t^j f^m) v_\mu=\mu(t^j f^m) v_{\mu}$.  

%We also note that (i) and (ii) agree if $m> n+k+1 - \ell$ or if $s_\ell>1$. Therefore, for most of the proof we assume $\ell \geq 1$. At the end of the proof, we treat the cases $\ell \ge 2$ and $\ell=1$ separately, under the assumptions $m=n+k+1 - \ell$ and $s_\ell=1$
We induct on $| \overline s|$.   For the base case, suppose  $| \overline s| =1$.  Then $f^{\overline s}=f^{\ell}$ and so 
$$
[t^j f^m, f^\ell ] v_\mu=-j \mu(t^j f^{m+\ell}) v_{\mu} + (\ell -m) \mu(t^{j+1}f^{m+ \ell -1})v_{\mu}.
$$
Since $m \geq n+r+1- \ell$, Lemma \ref{lem:degreehom} implies that $\mu(t^j f^{m+ \ell}) v_{\mu}=0$; and $\mu(t^{j+1}f^{m+ \ell -1})=0$ if $m>n+r+1- \ell$. This matches the claim.

\medskip

Now consider $\overline s  \in \z_{\ge 0}^n$ such that $|\overline s | >1$.  Using $f^{\overline s} = f^{\ell} \cdot f^{{\mathcal D}({\overline s})}$, we have

\begin{align*}
[ t^j f^m, f^{\overline s} ] v_\mu &= [t^j f^m, f^\ell ] \cdot f^{{\mathcal D}({\overline s)}} v_\mu + f^\ell \cdot [t^j f^m,  f^{{\mathcal D}({\overline s})} ] v_\mu \\
&= -j (t^j f^{m+ \ell} ) \cdot f^{{\mathcal D}({\overline s})} v_\mu  + (\ell - m) (t^{j+1} f^{m + \ell -1}) \cdot f^{{\mathcal D}({\overline s})} v_\mu \\
& \quad + f^\ell \cdot [t^j f^m, f^{{\mathcal D}({\overline s})} ] v_\mu \\
&= -j \mu (t^j f^{m+ \ell} ) f^{{\mathcal D}({\overline s})} v_\mu  + (\ell - m) \mu (t^{j+1} f^{m + \ell -1}) f^{{\mathcal D}({\overline s})} v_\mu\\
&-j [t^j f^{m+ \ell} , f^{{\mathcal D}({\overline s})}] v_\mu  + (\ell - m) [(t^{j+1} f^{m + \ell -1}),f^{{\mathcal D}({\overline s})}] v_\mu  \\
& \quad +  f^\ell \cdot [t^j f^m, f^{{\mathcal D}({\overline s})} ] v_\mu.
\end{align*}
 Since $m + \ell \ge n+r + 1$, Lemma \ref{lem:degreehom} again implies that $\mu(t^j f^{m+ \ell})=0$. Let $\ell'= \ell({\mathcal D}({\overline s}))$. Since $\ell' \geq \ell \ge 1$, we have that  $m + \ell \ge n+r + 1 > n+r + 1-\ell'$. It follows by induction that $[t^j f^{m+ \ell} , f^{{\mathcal D}({\overline s})}] v_\mu = 0$.  
We now have
\begin{align}\label{eqn:fm,fk-take1}
[ t^j f^m, f^{\overline s } ] v_\mu &=  (\ell - m) \mu (t^{j+1} f^{m + \ell -1}) f^{{\mathcal D}({\overline s})} v_\mu\\
&+ (\ell - m) [(t^{j+1} f^{m + \ell -1}),f^{{\mathcal D}({\overline s})}] v_\mu  +  f^\ell \cdot [t^j f^m, f^{{\mathcal D}({\overline s})} ] v_\mu. \nonumber
\end{align}

Suppose $m > n+r+1 - \ell$.  Then  $\mu(t^{j+1}f^{m+ \ell -1})=0$. Moreover, since $\ell' \ge \ell \ge 1$, we have $m > n+r+1 - \ell \ge n+r+1-\ell'$ and  $m + \ell - 1 > n+r \ge n+r+1-\ell'$. Therefore, by induction $[ t^{j+1} f^{m + \ell -1},  f^{{\mathcal D}({\overline s })}] v_\mu = [t^j f^m, f^{{\mathcal D}({\overline s })} ] v_\mu = 0$.
Applying this to (\ref{eqn:fm,fk-take1}), we get $[ t^j f^m, f^{\overline s } ] v_\mu = 0$ as desired.  

For the remainder of the proof,  we examine (\ref{eqn:fm,fk-take1}) under the assumption $m =n+r+1 - \ell$.  If $s_{\ell}=1$, then $\ell'>\ell$ and $[t^j f^{m}, f^{{\mathcal D}({\overline s })} ] v_\mu=[(t^{j+1} f^{m + \ell -1}),f^{{\mathcal D}({\overline s})}] v_\mu=0$ by induction.  Then, (\ref{eqn:fm,fk-take1}) becomes 
$$[ t^j f^m, f^{\overline s } ] v_\mu = (2 \ell -n-r-1) \mu ( t^{j+1} f^{n+r} ) f^{{\mathcal D}({\overline s })} v_\mu.$$ This matches the claim.

Since the result has been shown in the case that $s_\ell = 1$, we now suppose $s_\ell > 1$.  Note that $\ell'=\ell$. First consider $\ell \ge 2$. Then $m+\ell-1>n+r+1 - \ell'$; and  $[ t^{j+1} f^{m + \ell -1},  f^{{\mathcal D}({\overline s })}] v_\mu=0$ by induction. Applying the inductive hypothesis to $[t^j f^m, f^{{\mathcal D}({\overline s})} ] v_\mu$, the equation (\ref{eqn:fm,fk-take1}) becomes 
\begin{align*}
[ t^j f^m, f^{\overline s } ] v_\mu &=
(\ell-m) \mu ( t^{j+1} f^{n+r} ) f^{{\mathcal D}({\overline s })} v_\mu \\
&+ (\ell-m) (s_\ell -1) \mu ( t^{j+1} f^{n+r}) f^\ell \cdot
 f^{{\mathcal D}^2({\overline s})} v_{\mu} \\
&= (\ell-m)s_{\ell} \mu ( t^{j+1} f^{n+r} ) f^{{\mathcal D}({\overline s })} v_\mu 
\end{align*}
as desired.

It remains to consider the case $\ell = 1$ and $s_\ell > 1$.  Then $m = n+r$ and $\ell'=1$.  Applying the inductive hypothesis, (\ref{eqn:fm,fk-take1}) becomes 
\begin{align*}
[ t^j f^m, f^{\overline s } ] v_\mu &=  (1-m) \mu (t^{j+1} f^{m}) f^{{\mathcal D}({\overline s})} v_\mu \\
&+ (1-m) [t^{j+1} f^m,f^{{\mathcal D}({\overline s})}] v_\mu  +  f^1 \cdot [t^j f^{m}, f^{{\mathcal D}({\overline s})} ] v_\mu \\
&= (1-m) \mu (t^{j+1} f^{m}) f^{{\mathcal D}({\overline s})} v_\mu\\
&+\sum_{j=1}^{r_1}  (1 - m)^{j+1} {s_1-1 \choose j} \mu (t^{i+j+1} f^{m}) f^{{\mathcal D}^{j+1}({\overline s })}  v_\mu \\
&+\sum_{j=1}^{r_1}  (1 - m)^{j} {s_1-1 \choose j} \mu (t^{i+j} f^{m}) f^{{\mathcal D}^{j}({\overline s })}  v_\mu.
\end{align*}
This reduces to the formula asserted when $m = n+r$ and $\ell = 1$.
\end{proof}

\begin{lem}\label{lem:repRootPowerComp3}
Let $(\lambda, n, f; \mu, p, r)$ be as in Conventions \ref{conv}; and let ${\overline s} = (s_0, \ldots, s_{n-1}) \in \z_{\ge 0}^n$ with $s_0>0$. Let $j, m \in \z$ and suppose $m \ge n+r+s_0$. 
\begin{itemize}
\item[(i)] If $m> n+r+s_0$, then $[t^j f^{m}, f^{\overline s } ] v_\mu=0$. 
\item[(ii)] If  $m= n+r+s_0$, then 
\begin{align*}
[t^j f^{m}, f^{\overline s } ] v_\mu&=(-1)^{s_0} \frac{(n+r+s_0)!}{(n+k)!} \\
&\quad \times \left( \mu ( t^{j+s_0} f^{n+r} ) f^{\tilde D (\bar s)} v_\mu + [t^{j+s_0} f^{n+r} , f^{\tilde D (\bar s)}] \right) v_\mu
\end{align*}
where $\tilde D (\bar s) = (0, s_1, \ldots, s_{n-1})$.
\end{itemize}
%Then 
%$$[t^j f^{m}, f^{\overline s } ] v_\mu = 
%\left\{ 
%\begin{array}{ll}0 & \mbox{if $m> n+k+s_0$}\\
%(-1)^{s_0} \frac{(n+k+s_0)!}{(n+k)!} \left( \mu ( t^{i+s_0} f^{n+k} ) f^{\tilde s} v_\mu + [t^{i+s_0} f^{n+k} , f^{\tilde s}] v_\mu \right) & \mbox{if $m= n+k+s_0$} 
%\end{array} \right.$$ 
%{\color{magenta}
%ALTERNATIVE WAY TO WRITE FORMULA
%$$[t^j f^{m}, f^{\overline s } ] v_\mu = 
%\left\{ 
%\begin{array}{ll}0 & \mbox{if $m> n+k+s_0$}\\
%(-1)^{s_0} \frac{m!}{(m - s_0)!} \left( \mu ( t^{i+s_0} f^{n+k} ) f^{\tilde D (\bar s)} v_\mu + [t^{i+s_0} f^{n+k} , f^{\tilde D (\bar s)}] v_\mu \right) & \mbox{if $m= n+k+s_0$} 
%\end{array} \right.$$ 
%ANY BETTER?} \\
%where $\tilde D (\bar s)= (0, s_1, \ldots, s_{n-1})$.
\end{lem}
\begin{proof}
By Lemma \ref{lem:degreehom}, $\mu(t^j f^{m})=0$ for $m > n+r$.  Then using $f^{\overline s } = f^0 \cdot f^{{\mathcal D}({\overline s })}$, we have 
\begin{align}
[t^j f^{m} , f^{\overline s }] v_\mu &= (-j t^j f^{m} - m t^{j+1} f^{m-1}) \cdot f^{{\mathcal D}(\overline s)} v_\mu \label{eq:s_0} \\
&+ f^0 \cdot [t^j f^{m}, f^{{\mathcal D}(\overline s)}] v_{\mu}\nonumber \\
&=-j [t^j f^{m} , f^{{\mathcal D}(\overline s)}] v_\mu + f^0 \cdot [t^j f^{m}, f^{{\mathcal D}(\overline s)}] v_{\mu} \nonumber\\
&\quad - m \left( \mu(t^{j+1} f^{m-1}) f^{{\mathcal D}(\overline s)} + [t^{j+1} f^{m-1}, f^{{\mathcal D}(\overline s)}] \right)v_\mu.\nonumber 
\end{align}

The proof is by induction on $s_0$.  Suppose $s_0 = 1$. 
Applying Lemma \ref{lem:repRootPowerComp1} to (\ref{eq:s_0}), we obtain the result for $s_0=1$ since ${\mathcal D}( \overline s) = \tilde {\mathcal D} (\bar s)$.

Now assume that $s_0 > 1$.  Then $m-1 > n+r$ and thus $\mu(t^{j+1} f^{m-1})=0$ by Lemma \ref{lem:degreehom}.  Also, $[t^j f^{m}, f^{{\mathcal D}({\overline s })}] v_\mu = 0$ by induction. Thus (\ref{eq:s_0}) becomes
\begin{align}
[t^j f^{m}, f^{\overline s }] v_\mu 
&= - m  [t^{j+1} f^{m-1}, f^{{\mathcal D}({\overline s })}]  v_\mu. \label{eqn:s0case}
\end{align}
If $m>n+r+s_0$, then (\ref{eqn:s0case}) is zero by induction as desired.  
If $m=n+r+s_0$, then the inductive hypothesis applied to (\ref{eqn:s0case}) yields
\begin{align*}
[t^j f^{n+r+s_0}, f^{\overline s }] v_\mu
&= - (n+r+s_0)(-1)^{s_0-1} \frac{(n+r+s_0-1)!}{(n+r)!} \\
& \quad \times \left( \mu ( t^{i+1+s_0-1} f^{n+r} ) f^{\tilde {\mathcal D} (\bar s)} v_\mu + [t^{i+1+s_0-1} f^{n+r} , f^{\tilde {\mathcal D} (\bar s)}] \right) v_\mu .
\end{align*}
This matches the claim.
\end{proof}

For $0 \neq \sum_{\overline s} c_{\overline s} f^{\overline s} v_\mu \in V_\mu^{f^n}$, there is a unique maximal element $\overline u \in \z_{\ge 0}^n$, with respect to the lexicographic ordering, such that $c_{\bar u} \neq 0$. Define $\Lambda \left( \sum_{\overline s} c_{\overline s} f^{\overline s} v_\mu \right) = \overline u$.   

For the lemma below, note that if $\mu$ is a homomorphism of degree $k$ where $k>n-3$, then $\mu$ is necessarily nonzero whenever $n\geq 2$.
\begin{lem} \label{lem:reducedegree}
Let $(\lambda, n, f; \mu, p, r)$ be as in Conventions \ref{conv}, and suppose $\mu$ is nonzero with degree $r >n-3$.  Let $ \bar s\in \z_{\ge 0}^n$ with $| \bar s| >0$, and define
$$
m= \left\{ 
\begin{array}{ll}
n+r+1-\ell(\bar s) & \mbox{if $\ell(s)>0$;}\\
n+r+s_0 & \mbox{if $\ell(s)=0$.}
\end{array} \right.
$$
Then
\begin{itemize}
\item[(i)]
for all but finitely many $j \in \z$,  $(t^j f^m-\mu(t^j f^m)) \cdot ( f^{\bar s} v_{\mu}) \neq 0$ and $\Lambda ((t^jf^m-\mu(t^jf^m)) \cdot ( f^{\bar s} v_{\mu})) =\mathcal D ({\bar s})$;
\item[(ii)] for any $\bar s' \in \z_{\ge 0}^n$ with $\bar s' \prec \bar s$, either $(t^jf^m-\mu(t^jf^m)) \cdot ( f^{\bar s'} v_{\mu})=0$ or
$
\Lambda ((t^jf^m-\mu(t^jf^m)) \cdot ( f^{\bar s'} v_{\mu})) \prec \mathcal D ({\bar s})$.
\end{itemize}
%\begin{itemize}
%\item $(t^j f^m-\mu(t^j f^m)) ( f^{\bar s} v_{\mu}) \neq 0$ and $\Lambda ((t^jf^m-\mu(t^jf^m)) ( f^{\bar s} v_{\mu})) \prec {\bar s}$;
%\item $\Lambda ((t^jf^m-\mu(t^jf^m)) ( f^{\bar s'} v_{\mu})) \prec \Lambda ((t^jf^m-\mu(t^jf^m)) ( f^{\bar s} v_{\mu}))$ for any ${\bar s'} \prec {\bar s}$.
%\end{itemize}
\end{lem}

\begin{proof}
Note that $0 \leq \ell(\bar s)\leq n-1$. Since $r \geq n-2$, it follows that $m=n+r+1-\ell \geq n+(n-2)+1-(n-1)=n$. Thus, $m \geq n$ and $\mu(t^jf^m)$ is defined.  For the proof, we consider ${\bar s}' \preceq {\bar s}$, so that we can address claims (i) and (ii) simultaneously.

If ${\bar s}' = (0, \ldots, 0)$, then clearly $(t^jf^m-\mu(t^jf^m)) \cdot ( f^{\bar s'} v_{\mu})=0$.
Now suppose ${\bar s}'  \neq (0, \ldots, 0)$. Note that
\begin{align}
(t^jf^m-\mu(t^jf^m)) \cdot ( f^{\bar s'} v_{\mu})=[t^jf^m,  f^{\bar s'}] v_\mu. 
\label{eqn:reduceddegreebracket}
\end{align}
 Write $\bar s= (s_0, \ldots, s_{n-1})$, $\bar s'= (s_0', \ldots, s_{n-1}')$; and $\ell = \ell(\bar s)$, $\ell'=\ell(\bar s')$. Since $\bar s' \preceq \bar s$, it follows that $\ell' \geq \ell$ and $s_\ell' \leq s_\ell$. 

Suppose $\ell > 0$, so that $m=n+r+1-\ell$.  
%\begin{align}
%(t^j f^{n+r+1 - \ell} - \mu (t^j f^{n+r+1 - \ell})) (f^{{\bar s}'} v_\mu)  
%&=  [t^j f^{n+r+1 -\ell},  f^{{\overline s }'}] v_\mu. \label{eqn:reduceddegreebracket}
%\end{align}
Applying
Lemma \ref{lem:repRootPowerComp1} to (\ref{eqn:reduceddegreebracket}), we have the following:
\begin{itemize}
%\item if ${\bar s}' = \bar s$, then the maximal nonzero term in (\ref{eqn:reduceddegreebracket}) is $f^{\mathcal D ({\bar s})} v_\mu $;
\item If $\ell'=\ell$, the maximal term in (\ref{eqn:reduceddegreebracket}) is $f^{\mathcal D ({\bar s}')} v_\mu$, with coefficient $( \ell - m ) s_\ell'  \mu(t^{j+1} f^{n+r})$. (If $\ell(\bar s)>1$, this is the only term.)  
\item If $\ell'>\ell$, then (\ref{eqn:reduceddegreebracket}) is zero. 
\end{itemize}
By Lemma \ref{lem:degreehom}, $\mu(t^{j+1} f^{n+r}) \neq 0$ for all but finitely many $j \in \z$. Since 
$\mathcal D ({\bar s}') \prec \mathcal D ({\bar s})$ for $\bar s' \prec \bar s$,
 this proves both claims for $\ell>0$.

Suppose now that $\ell = 0$.    Applying Lemmas \ref{lem:repRootPowerComp3} and \ref{lem:repRootPowerComp1} to (\ref{eqn:reduceddegreebracket})
\begin{itemize}
\item if $s_0'=s_0$, then the maximal term in (\ref{eqn:reduceddegreebracket}) is $f^{{\tilde {\mathcal D}} ({\bar s}')} v_\mu$, with coefficient $\mu(t^{i+s_0} f^{n+k})$;
\item if $s_0'<s_0$, then (\ref{eqn:reduceddegreebracket}) is zero. 
\end{itemize}
 By Lemma \ref{lem:degreehom}, $\mu(t^{j+s_0} f^{n+r}) \neq 0$ for all but finitely many $j \in \z$; this completes the proof the proof for $\ell=0$.
\end{proof}

\begin{lem}\label{lem:brack-tupleSize}
Let $(\lambda, n, f; \mu, p, r)$ be as in Conventions \ref{conv} and $\bar s \in \z_{\ge 0}^n$ with $|\bar s| > 0$.  Then for $j, m \in \z$ with $m \geq n+s_0$, 
$$[t^j f^m, f^{\overline s}] v_\mu \in {\rm span}_\cc \{ f^{\overline r} v_\mu \mid | \overline r| < | \overline s |  \}.$$
\end{lem}

\begin{proof}
We prove this by induction on $| \bar s|$. The base case follows immediately from the commutator relation.  
Now suppose $|\bar s| >1$ and recall $f^{\overline s} =  f^{\ell (\overline s)} \cdot f^{{\mathcal D}(\overline s)}$.  Then, 
\begin{align}
[t^j f^m, f^{\overline s}] v_\mu &= [t^j f^m, f^{\ell (\overline s)} \cdot f^{{\mathcal D}(\overline s)} ] v_\mu \nonumber \\
&= [t^j f^m, f^{\ell (\overline s)} ] \cdot f^{{\mathcal D}(\overline s)} v_\mu + f^{\ell (\overline s)} \cdot [t^j f^m, f^{{\mathcal D}(\overline s)} ] v_\mu \nonumber \\
&= \left(  -jt^j f^{m + \ell ( \overline s)} + (\ell ( \overline s) - m) t^{j+1} f^{m + \ell ( \overline s)-1}  \right) \cdot f^{{\mathcal D}(\overline s)} v_\mu \nonumber \\
& \quad + f^{\ell (\overline s)} \cdot [t^j f^m, f^{{\mathcal D}(\overline s)} ] v_\mu \nonumber \\
&= -j \mu(t^j f^{m + \ell ( \overline s)}) f^{{\mathcal D}(\overline s)} v_\mu + (\ell ( \overline s) - m) \mu(t^{j+1} f^{m + \ell ( \overline s)-1} ) f^{{\mathcal D}(\overline s)} v_\mu \label{eqn:bracket2}\\
&\quad 
-j  [t^j f^{m + \ell ( \overline s)} , f^{{\mathcal D}(\overline s)}] v_\mu + (\ell ( \overline s) - m) [t^{j+1} f^{m + \ell ( \overline s)-1} , f^{{\mathcal D}(\overline s)}] v_\mu \label{eqn:bracket1}\\
& \quad + f^{\ell (\overline s)} \cdot [t^j f^m, f^{{\mathcal D}(\overline s)} ] v_\mu \label{eqn:bracket3}
\end{align}
Since  $|{\mathcal D}(\overline s)| < |\bar s|$, the terms in (\ref{eqn:bracket2}) have the desired form. Also, if $\ell(\bar s)=0$, then ${\mathcal D}(\overline s)_0< s_0$; this implies $m+ \ell (s), m+\ell(s)-1 \geq n+ {\mathcal D}(\overline s)_0$. Therefore, the terms in (\ref{eqn:bracket1}) have the desired form by induction.

Applying the inductive hypothesis to (\ref{eqn:bracket3}), we obtain
\begin{align*}
f^{\ell (\overline s)} \cdot [t^j f^m, f^{{\mathcal D}(\overline s)} ] v_\mu &= \sum_{| \overline r | < | {\mathcal D}( \overline s)|} c_{\overline r} f^{\ell ( \overline s )} \cdot f^{\overline r} v_\mu\\
&=\sum_{| \overline r' | \leq  | {\mathcal D}( \overline s) |} c_{\overline r'}'  \cdot f^{\overline r'} v_\mu
\end{align*}
where $c_{\overline r}, c_{\overline r'}'  \in \cc$ and we apply a standard straightening argument to go from the first to the second line. Thus, (\ref{eqn:bracket3}) has the form given in the claim. 

%{\color{blue}MO Comment: I still have a couple questions here, but I'm confident in the basic idea of the lemma.}
\end{proof}

\subsection{$V^{f^n}_{\mu}$ is not simple for homomorphisms $\mu$ of small degree}
With $(\lambda, n, f; \mu, p, r)$ as in Conventions \ref{conv}, we show that if  $r\le n-3$, then $V^{f^n}_{\mu}$ is not simple. First we establish a computational lemma.

\begin{lem} \label{lem:Faulhaber}
Let $j, k \in \z_{>0}$.  Define $P_k(t) = \frac{1}{k+1}\sum_{i=0}^k (-1)^i \binom{k+1}{i} B_\ell t^{k+1-i}$, where $B_i$ denotes the $i^{th}$ Bernoulli number . Then,
\begin{itemize}
\item[(i)] (Faulhaber's Formula) $\sum_{i=1}^j i^k = P_k(j)$;
\item[(ii)] $\sum_{i=1}^j (-i)^k = -P_k(-j-1)$. 
\end{itemize}
\end{lem}

\begin{proof} 
The statement in (i) is a version of the well-known formula for sums of powers of positive integers; for example, see \cite{B96}. To prove (ii), recall that $B_1=-1/2$ and $B_{i}=0$ for $i \neq 1$ odd. Then, $P_k(-t)= (-1)^{k+1} \left( P_k(t) - t^k\right)$, and so 
\begin{align*}
P_k(-(j+1)) &=(-1) \left( P_k(j+1)-(j+1)^k \right) \\
&=(-1)^{k+1} \left(\sum_{i=1}^{j+1} - (i+1)^k \right) \\
&= (-1)^{k+1} \sum_{i=1}^j i^k = - \sum_{i=1}^j (-i)^k.
\end{align*}
\end{proof} 

\begin{prop} \label{prop:notsimplesmalldegree}
Let $(\lambda, n, f; \mu, p, r)$ be as in Conventions \ref{conv}; and suppose $r \leq n-3$ or $n=1$ and $r=0$. Then
$$
M= {\rm span} \{ (f_{\lambda}^0)^{[s_0]} (f_{\lambda}^1)^{[s_1]} \cdots (f_{\lambda}^{n-1})^{[s_{n-1}]} v_{\mu} \mid s_0, \ldots, s_{n-2} \in \z_{\ge 0}, s_{n-1} \in \z_{\ge 1} \}
$$
is a proper submodule of $V_\mu^{f^n}$ and
$$
V_\mu^{f^n} /M \cong V_{\mu'}^{f^{n-1}},
$$
 where 
\begin{itemize}
\item if $\mu$ is the zero homomorphism, then $\mu'$ is the zero homomorphism;
\item if $r>0$, then $\mu': \Vir^{f^{n-1}} \rightarrow \cc$ is a homomorphism such that $\mu'(t^j f^{n-1}) = q(i) \lambda^i$ for some polynomial $q(i)$ of degree $r+1$.
\end{itemize}
\end{prop}

\begin{proof}
We have
\begin{align}
t^j f^n \cdot f^{n-1} v_{\mu} &= (f^{n-1} \cdot t^j f^n +[t^j f^n, f^{n-1}]) v_{\mu} \nonumber \\
&= \mu(t^j f^n) f^{n-1} v_{\mu} \nonumber \\
&\quad + (-i\mu(t^j f^{2n-1}) -\mu(t^{j+1} f^{2n-2})) v_{\mu}.
\end{align}
By our assumption $r \leq n-3$, $\mu(t^j f^{2n-1})=\mu(t^{j+1} f^{2n-2}))=0$.  Thus $f^{n-1} v_{\mu}$ generates a one-dimensional $\langle f^n \rangle$-module isomorphic to $\cc_{\mu}$. Let $M = U( \Vir ) \cdot f^{n-1} v_\mu$ be the $\Vir$-module generated by $f^{n-1} v_{\mu}$.  Then the universal property of $V_\mu^{(t - \lambda)^n}$ implies that there is a surjective map $V_\mu^{(t - \lambda)^n} \rightarrow M$. By applying an appropriate PBW basis of $U(\Vir)$ (see the proof of Lemma \ref{lem:inducedbases}) we get that $M$ has a basis
$$
\{ (f_{\lambda}^0)^{[s_0]} (f_{\lambda}^1)^{[s_1]} \cdots (f_{\lambda}^{n-1})^{[s_{n-1}]} \cdot f^{n-1} v_{\mu} \mid s_0, \ldots, s_{n-1} \in \z_{\geq 0} \}.
$$
This shows that $M$ is a proper submodule of $V_\mu^{f^n}$ and $M \cong V_\mu^{f^n}$ via the map described above. 

Now we consider $V_\mu^{f^n} /M$.  Let $\overline v_{\mu}$ be the image of $v_{\mu}$ in $V_\mu^{f^n} /M$. Note that $f^{n-1} \overline v_{\mu}=0$ and $t^j f^n \overline v_\mu = \mu(t^j f^n) \bar v_\mu=p(j) \lambda^j \bar v_\mu$ for all $j \in \z$. Using the relation $t^j f^{n-1}-\lambda t^{i-1} f^{n-1}=t^{j-1} f^n$, it's straightforward to argue inductively that
\begin{align*}
t^jf^{n-1} \overline v_{\mu} &= \lambda^{j-1} \sum_{i=0}^{j-1} p(i) \overline v_{\mu}; \\
t^{-j}f^{n-1} \overline v_{\mu} &= -\lambda^{-j-1} \sum_{i=1}^{j} p(-i) \overline v_{\mu}.
\end{align*}
for $j \in \z_{\ge 1}$. 

It follows from Lemma \ref{lem:Faulhaber} that $t^j f^{n-1} \overline v_{\mu}=\mu'(t^j f^{n-1}) \overline v_{\mu}$ where $\mu'$  is a homomorphism of the form claimed.

The universal property for $V_{\mu'}^{f^{n-1}}$ implies that we have a surjective homomorphism $\theta: V_{\mu'}^{f^{n-1}} \rightarrow V_\mu^{f^n} /M$ such that $\theta(v_{\mu'})= \bar v_{\mu}$ .  Given the basis for $M$ and the basis for $V_\mu^{f^n}$, it follows that $\{ (f^0)^{[k_0]} \cdots (f^{n-2})^{[s_{n-2}]} \overline v_{\mu} \mid s_0, \ldots, s_{n-2} \in \z_{\geq 0} \}$ is a basis for $V_\mu^{f^n} /M$. Thus, $\theta$ maps the standard basis for  $V_{\mu'}^{f^{n-1}}$ to the basis for $V_\mu^{f^n} /M$. It follows that $\theta$ is an isomorphism.
\end{proof}

\section{Tensor products of induced modules} \label{sec:tensorproducts}

For $L \in \z_{\geq 0}$, define $\Vir_{\geq L}= {\rm span} \{ z, e_j \mid j \geq L \}$. We refer to a $\Vir$-module $V$ as {\it $\Vir^+$-locally annihilated} if for each $v \in V$, there is some $L \in \z_{\geq 0}$ such that $\Vir_{\geq L} v=0$.  In this section we study tensor products of the form
\begin{equation} \label{eqn:tensor}
\bigotimes_{i=1}^k V^{(t-\lambda_i)^{n_i}}_{\mu_i} \otimes V,
\end{equation}
where $V$ is a cyclic, $\Vir^+$-locally annihilated module. 
In particular, we determine simplicity conditions for tensor products of the form  (\ref{eqn:tensor}) and conditions for when two such tensor products can be isomorphic. In Section \ref{sec:tensorapplications}, we specialize $V$--to Verma modules and certain quotients, as well as Whittaker modules--to produce new classes of simple induced modules. 

This tensor product construction is inspired by and generalizes \cite{TZ16, TZ13}. Moreover, \cite{MZ14} have shown that simple modules $V$ with a $\Vir_N$-locally finite action, for some $N \in \z_{\geq 0}$, have the form for $V$ described above.

We use the following conventions throughout the section. Recall that we regard $V_\mu^{f_i^{n_i}}$ simultaneously as a $\Vir$-module and as a $\cc [t^{\pm}]$-module.
 \begin{conv} \label{conv2}
Fix $k \in \z_{\ge 1}$, $n_1, \ldots, n_k \in \z_{\geq 1}$ and $\lambda_1, \ldots, \lambda_k \in \cc^\times$ such that $\lambda_i \neq \lambda_j$ for $i \neq j$. For each $1 \leq i \leq k$, let $f_i = t-\lambda_i$, and let $\mu_i:  \Vir^{f_i^{n_i}}  \rightarrow \cc$ be a homomorphism associated to the polynomial $p_i$ of degree $r_i$ as in Section \ref{sec:homsAndOneDimReps}.  We summarize this as $(k; \lambda_i, n_i, f_i; \mu_i, p_i, r_i)$.   
Define $n_0 = \sum_{i=1}^k n_i$. Let $v_i \in V^{f_i^{n_i}}_{\mu_i}$ represent the canonical generator of $V^{f_i^{n_i}}_{\mu_i}$ and let $v_0 = v_1 \otimes \cdots \otimes v_k \in \bigotimes_{i=1}^k V^{f_i^{n_i}}_{\mu_i}$.
 \end{conv}
  
 To denote bases for $\bigotimes_{i=1}^k V^{f_i^{n_i}}_{\mu_i}$, note that any element ${\bar {\bar s}} \in \z_{\ge 0}^{n_0}$ can be expressed
${\bar {\bar s}}= (\bar s_1, \ldots, \bar s_k) = (s_{1, 0} \, \ldots, s_{1, n_1-1}, \ldots, s_{k, 0} \, \ldots, s_{k, n_k-1})$,
where $\overline s_i = ( s_{i,0}, \ldots , s_{i,n_i - 1})$; let $|{\bar {\bar s}}|= \sum_{i=1}^k |\bar s_i |$. Define
\begin{align*}
f^{\bar {\bar s}} v_0 &= f_1^{\bar s_1} v_1 \otimes \cdots \otimes f_k^{\bar s_k} v_k.
%\\{\color{green}t^{\bar {\bar s}} v_0} &= t^{\bar s_1} v_1 \otimes \cdots \otimes t^{\bar s_k} v_k.
\end{align*}
Lemma \ref{lem:inducedbases} implies that the set
$\{ f^{\bar {\bar s}} v_0 \mid {\bar {\bar s}} \in \z_{\ge 0}^{n_0}\}$ is a basis for $\bigotimes_{i=1}^k V^{f_i^{n_i}}_{\mu_i}$. Below we apply the lexicographic ordering to $\z_{\ge 0}^{n_0}$; this is compatible in the obvious way with the lexicographic order on individual components of the tensor product.

Regarding $\cc [t]$ as a Lie subalgebra of $\cc [t^\pm]$, Lemma \ref{lem:inducedbases} implies that $V_{\mu_i}^{f_i^{n_i}}$ is generated as a $\cc [t]$-module by the vector $v_i$.  The following lemma shows that $ \bigotimes_{i=1}^k V^{f_i^{n_i}}_{\mu_i}$ is generated as a $\cc[t]$-module by $v_0$. 

\begin{lem} \label{lem:cyclictensor}
Let $(k; \lambda_i, n_i, f_i; \mu_i, p_i, r_i)$ be as in Conventions \ref{conv2}. Then,
$$\bigotimes_{i=1}^k V^{f_i^{n_i}}_{\mu_i} = U(\cc [t]) v_0.$$
\end{lem}
\begin{proof}
To prove this, it is enough to show that $f^{\bar {\bar s}} v_0 \in U(\cc [t]) v_0$ for any ${\bar {\bar s}}  \in \z_{\ge 0}^{n_0}$. We argue by induction on $|{\bar {\bar s}} |$, noting that the base case $|{\bar {\bar s}} |=0$ is trivial. Now let  ${\bar {\bar s}}  \in \z_{\ge 0}^{n_0}$ with $|{\bar {\bar s}} |>0$ and assume $f^{\bar {\bar s}'} v_0 \in U( \cc [t]) v_0$ for any ${\bar {\bar s}'} \in \z_{\ge 0}^{n_0}$ with $|{\bar {\bar s}'}|<|{\bar {\bar s}}|$.

Write ${\bar {\bar s}} = (\bar s_1, \ldots, \bar s_k)$.  Choose any $j \in \{ 1, \ldots , k \}$ such that $| \bar s_j |\neq 0$, and write $\ell=\ell(\bar s_j)$ as in Section \ref{subsec:basisNotation}.  Since the polynomials $f_j^{n_j + s_{j,0}}$ and $F_{\hat j}=\prod_{j' \neq j} f_{j'}^{n_{j'} + s_{j',0}}$ are relatively prime in $\cc[t]$, there exist $g_j, g_{\hat j} \in \cc[t]$ such that $g_j f_j^{n_j + s_{j,0}}+ g_{\hat j}F_{\hat j}=f_j^{\ell}$.  Let $\mathcal D_j ({\bar {\bar s}}) = (\bar s_0, \ldots, \mathcal D (\bar s_j), \ldots, \bar s_k)$.  Then 
\begin{align*}
 g_{\hat j}F_{\hat j} \cdot f^{\mathcal D_j ({\bar {\bar s}})} v_0 &= \sum_{j' \neq j}  \cdots \otimes g_{\hat j}F_{\hat j} \cdot f_{j'}^{\bar s_{j'}} v_{j'} \otimes \cdots  \nonumber \\
 &\quad +  \cdots \otimes (f_j^{\ell}-g_j f_j^{n_j + s_{j,0}}) \cdot f_{j}^{\mathcal D (\bar s_j)} v_j \otimes \cdots \nonumber \\
 &=\sum_{j' \neq j}  \cdots \otimes \mu_{j'}(g_{\hat j}F_{\hat j}) f_{j'}^{\bar s_{j'}} v_{j'} \otimes \cdots  +  \sum_{j' \neq j}\cdots \otimes [g_{\hat j}F_{\hat j}, f_{j'}^{\bar s_{j'}}] v_{j'} \otimes \cdots  \\
 &\quad - \cdots \otimes \mu_j(g_j f_j^{n_j +s_{j,0}}) f_{j}^{\mathcal D (\bar s_j)} v_j \otimes \cdots - \cdots \otimes [g_j f_j^{n_j+s_{j,0}}, f_{j}^{\mathcal D (\bar s_j)}] v_j \otimes \cdots  \\
 &\quad + \cdots \otimes f_j^{\ell} \cdot f_{j}^{D (\bar s_j)} v_j \otimes \cdots \\
 &= X + f^{\bar {\bar s}} v_0,
 \end{align*}
 where, using Lemma \ref{lem:brack-tupleSize},  $X \in {\rm span} \{ f^{\bar {\bar s}'} v_0 \mid |{\bar {\bar s}'}| <|{\bar {\bar s}}| \}$. It follows that $f^{\mathcal D_j ({\bar {\bar s}})} v_0, X  \in U(\cc[t]) v_0$ by the inductive hypothesis; this completes the proof. 
\end{proof}

The next two lemmas allow us to extend Lemma \ref{lem:cyclictensor} to a tensor product involving a cyclic $\Vir^+$-locally annihilated module.

\begin{lem} \label{lem:UUk}
Let $(k; \lambda_i, n_i, f_i; \mu_i, p_i, r_i)$ be as in Conventions \ref{conv2}. Suppose $L \in \z_{\ge 0}$, and define $\cc[t]_L= {\rm span} \{ t^j \mid j \geq L\}$. For any $w \in \bigotimes_{i=1}^k V_{\mu_i}^{f_i^{n_i}}$ and $h \in \cc[t]$, there exists $\tilde h \in \cc [t]_L$ such that $\tilde h w =h w$.
\end{lem}
\begin{proof}
Fix $h \in \cc [t]$.  By Lemmas \ref{lem:repRootPowerComp1} and \ref{lem:repRootPowerComp3}, we can choose $N_i \in \z_{\ge 1}$, $1 \leq i \leq k$, such that $\langle \prod_{i=1}^k f_i^{N_i} \rangle w=0$. Define $F=\prod_{i=1}^k f_i^{N_i}$.
Since $\lambda_i \neq 0$, the polynomials $t^L$ and $F$ are relatively prime in $\cc [t]$. Thus, there exist $g_1, g_2 \in \cc [t]$ such that $h = g_1 t^L + g_2F$.  It follows that 
$$hw = (g_1 t^L + g_2 F) w = (g_1 t^L) w.$$
Then $\tilde h:= g_1 t^L \in \cc[t]_L$. 
\end{proof}

\begin{lem} \label{lem:WVgen}
Let $(k; \lambda_i, n_i, f_i; \mu_i, p_i, r_i)$ be as in Conventions \ref{conv2}; and
 suppose $V$ is a cyclic $\Vir^+$-locally annihilated module with generator $v \in V$.  Then the module $\bigotimes_{i=1}^k V^{f_i^{n_i}}_{\mu_i} \otimes V$ is cyclic with generator $v_0 \otimes v$.
\end{lem}
\begin{proof}
Let $M = U(\Vir) (v_0 \otimes v)$, and let $L\in \z_{\ge 0}$ such that $\Vir_{\geq L} v=0$.  Lemmas \ref{lem:cyclictensor} and \ref{lem:UUk} imply that $\bigotimes_{i=1}^k V^{f_i^{n_i}}_{\mu_i}$ is generated by $v_0$ as a $\cc[t]_{\geq L}$-module. It follows that $\bigotimes_{i=1}^k V^{f_i^{n_i}}_{\mu_i} \otimes v \subseteq M$.  

Let $X$ be the maximal subspace of $V$ such that $\bigotimes_{i=1}^k V^{f_i^{n_i}}_{\mu_i} \otimes X \subseteq M$. Note that $X$ must be a $\Vir$-submodule of $V$. Since $v \in X$ generates $V$, it follows that $X=V$. Thus, $\bigotimes_{i=1}^k V^{f_i^{n_i}}_{\mu_i} \otimes V = M$. 
\end{proof}

\begin{thm} \label{thm:tensorsimpleV}
Let $(k; \lambda_i, n_i, f_i; \mu_i, p_i, r_i)$ be as in Conventions \ref{conv2} and suppose $V$ is a $\Vir^+$-locally annihilated module. Then the module 
$$
\bigotimes_{i=1}^k V_{\mu_i}^{(t-\lambda_i)^{n_i}} \otimes V
$$ 
is simple if and only if both $V$ is simple and $r_i \geq n_i-2$ for each $1 \leq i \leq k$.
\end{thm}
\begin{proof}
If $r_i < n_i-2$ for some $1 \leq i \leq k$, then $V_{\mu_i}^{f_i^{n_i}}$ is not simple by Proposition \ref{prop:notsimplesmalldegree}. If any factor in the tensor product is not simple, clearly the resulting tensor product is not simple. This completes the ``only if" direction.

Now suppose $V$ is simple and $r_i \geq n_i-2$ for each $1 \leq i \leq k$. Since $V$ is simple, it follows that $V$ is cyclic and generated by any $0 \neq v \in V$. Thus Lemma \ref{lem:WVgen} implies that $\bigotimes_{i=1}^k V_{\mu_i}^{f_i^{n_i}} \otimes V$ is generated by $v_0 \otimes v$.

Let $0 \neq w \sum_{\bar {\bar s} \in \z_{\ge 0}^{n_0}} f^{\bar {\bar s}} v_0 \otimes v_{\bar {\bar s}} \in \bigotimes_{i=1}^k V_{\mu_i}^{f_i^{n_i}} \otimes V$; and define $\Lambda(w) = {\bar {\bar u}}$ where ${\bar {\bar u}}$ is maximal (ordered lexicographically) such that $v_{{\bar {\bar u}}}\neq 0$.  To argue that $w$ generates the full module, it is enough to show that if $w \in \left( \bigotimes_{i=1}^k V_{\mu_i}^{f_i^{n_i}} \otimes V \right) \setminus v_0 \otimes V$, there is $0 \neq w' \in  U(\Vir) w$ such that $\Lambda(w') \prec \Lambda(w)$.  

Let $w = \sum_{\bar {\bar s} \in \z_{\ge 0}^{n_0}} f^{\bar {\bar s}} v_0 \otimes v_{\bar {\bar s}} \in \left( \bigotimes_{i=1}^k V_{\mu_i}^{f_i^{n_i}} \otimes V \right) \setminus v_0 \otimes V$, and set $\Lambda(w)=\bar {\bar u}=(\bar u_1, \ldots, \bar u_k)$. 
Choose $i_0$ minimal such that $| \overline{u}_{i_0}|>0$.  For each $1 \leq i \leq k$, we may apply Lemmas \ref{lem:degreehom}, \ref{lem:repRootPowerComp1} and \ref{lem:repRootPowerComp3} to choose $N_i \in \z_{\ge 1}$ such that, for all $\bar {\bar s}= (\bar s_1, \ldots, \bar s_k)$ with $v_{\bar {\bar s}} \neq 0$, we have $\langle f^{N_i} \rangle ( f_i^{\bar s_i} v_{\mu_i})=[\langle f^{N_i} \rangle, f_i^{\bar s_i}] v_{\mu_i}=0$.   
  Define $F_{i_0} = f_{i_0}^{N_{i_0}}$ and $F_{\hat i_0} = \prod_{i\neq i_0} f_{i}^{N_{i}}$. Now choose $m \in \z$ to satisfy the claims of Lemma \ref{lem:reducedegree} with respect to $f^{\overline{u}_{i_0}} v_{i_0}$.  Since $F_{i_0}$ and $F_{\hat i_0}$ are relatively prime in $\cc[t]$, there are $g_{i_0}, g_{\hat i_0} \in \cc[t]$ such that $f_{i_0}^{m}= g_{i_0} F_{i_0} + g_{\hat i_0} F_{\hat i_0}$. 

Since only finitely many $v_{{\bar {\bar s}}}$ are nonzero, we may choose $L\in \z_{\ge 0}$ such that $ \Vir_{\geq L} v_{{\bar {\bar s}}} =0$ for all corresponding ${\bar {\bar s}} \in \z_{\ge 0}^{n_0}$. By Lemma \ref{lem:reducedegree}, there is $j \in \z$ such that  $(t^j f^{m}_{i_0} - \mu_{i_0} (t^j f^{m}_{i_0})) f^{\overline{s}_{i_0}} v_{i_0} \neq 0$ and  $t^j g_{\hat i_0} F_{\hat i_0} \in \Vir_{\geq L}$. Write $h=t^jf_{i_0}^{m}$. Since $t^j g_{\hat i_0} F_{\hat i_0} v_{{\bar {\bar s}}}=0$ for all ${\bar {\bar s}} \in \z_{\ge 0}^{n_0}$, we have
\begin{align}
(t^j g_{\hat i_0} F_{\hat i_0} -\mu_{i_0}(h)) w&= -\mu_{i_0}(h) w  + \sum_{\bar {\bar s}}  \sum_{i =1}^k \cdots \otimes t^j g_{\hat i_0}  F_{\hat i_0} \cdot f_i^{\bar s_i} v_{\mu_i} \otimes \cdots \otimes v_{\bar {\bar s}}\nonumber \\
&=\sum_{\bar {\bar s}}  \sum_{i \neq i_0 } \cdots \otimes t^j g_{\hat i_0} F_{\hat i_0} \cdot f_i^{\bar s_i} v_{\mu_i} \otimes \cdots \otimes v_{\bar {\bar s}} \label{eqn:thm1} \\
&+ \sum_{\bar {\bar s}} \cdots \otimes (-t^j g_{i_0} F_{i_0} ) \cdot f_{i_0}^{\bar s_{i_0}} v_{i_0} \otimes \cdots \otimes v_{\bar {\bar s}} \label{eqn:thm2} \\
&+ \sum_{\bar {\bar s} \prec \bar {\bar u}} \cdots \otimes (h-\mu_{i_0}(h) ) \cdot f_{i_0}^{\bar s_{i_0}} v_{i_0} \otimes \cdots \otimes v_{\bar {\bar s}} \label{eqn:thm3} \\
&+  \cdots \otimes (h-\mu_{i_0}(h) ) \cdot f_{i_0}^{\bar u_{i_0}} v_{i_0} \otimes \cdots \otimes v_{\bar {\bar u}} \label{eqn:thm4}
\end{align}
Then (\ref{eqn:thm1}) and (\ref{eqn:thm2}) are zero by the definition of $F_{\hat i_0}$ and $F_{i_0}$. By Lemma \ref{lem:reducedegree} and our choice of $m_0$ and $j$, (\ref{eqn:thm4})  is nonzero and with maximal term indexed by $(0, \ldots, 0, \mathcal D (\bar u_{i_0}), \bar u_{i_0+1}, \ldots, \bar u_k)$. Lemma \ref{lem:reducedegree}, and the choice of $\bar {\bar u}$, also imply that the maximal term in (\ref{eqn:thm3}) is strictly smaller than (\ref{eqn:thm4}). This completes the proof.
\end{proof}

\begin{cor}\label{cor:tensIndPoly}
Let $(k; \lambda_i, n_i, f_i; \mu_i, p_i, r_i)$ be as in Conventions \ref{conv2}.  Then the module 
$$
\bigotimes_{i=1}^k V_{\mu_i}^{(t-\lambda_i)^{n_i}}
$$ 
is simple if and only if $r_i \geq n_i-2$ for each $1 \leq i \leq k$.
\end{cor}

We now consider when two such tensor products are isomorphic.

\begin{prop}\label{prop:whenIsomorphic}
Let $\lambda_1, \ldots, \lambda_k \in \cc^\times$ be distinct, $n_1, \ldots, n_k \in \z_{\ge 1}$, and $f_i= t-\lambda_i$. For each $1 \leq i \leq k$, let $\mu_i: \Vir^{f_i^{n_i}} \rightarrow \cc$ be a homomorphism. Similarly, let $\gamma_1, \ldots, \gamma_\ell \in \cc^\times$ be distinct, $m_1, \ldots, m_\ell \in \z_{\ge 1}$, and $g_i= t-\gamma_i$. For each $1 \leq i \leq \ell$, let $\theta_i: \Vir^{g_i^{m_i}} \rightarrow \cc$ be a homomorphism. Finally, let $V, W$ be a simple $\Vir^+$-locally annihilated $\Vir$ modules.  Then 
$$
\bigotimes_{i=1}^k V_{\mu_i}^{f_i^{n_i}} \otimes V \cong \bigotimes_{i=1}^\ell V_{\theta_i}^{g_i^{m_i}} \otimes W
$$
if and only if $V \cong W$, $k=\ell$, and after possibly renumbering, $\lambda_i = \gamma_i$, $n_i = m_i$, and $\mu_i = \gamma_i$ for all $i$. 
\end{prop}

\begin{proof}
Suppose $\phi: \bigotimes_{i=1}^k V_{\mu_i}^{f_i^{n_i}} \otimes V \rightarrow  \bigotimes_{i=1}^\ell V_{\theta_i}^{g_i^{m_i}} \otimes W$ is an isomorphism. 
Fix $0 \neq v \in V$ and write  
$$\phi (v_0 \otimes v) = \sum_{\bar {\bar s} \in \z_{\ge 0}^{m_0}} g^{\bar {\bar s}} w_0 \otimes w_{\bar {\bar s}}$$
for some $w_{\bar {\bar s}} \in W$. Since $w_{\bar {\bar s}} \neq 0$ for at most finitely many ${\bar {\bar s}}\in \z_{\ge 0}^{m_0}$, we may choose $L \in \z_{\ge 1}$ so that $\Vir_{\geq L} v=0$ and $\Vir_{\geq L} w_{\bar {\bar s}} = 0$ whenever $w_{\bar {\bar s}} \neq 0$.

Suppose $\{ \lambda_1, \ldots, \lambda_k \} \neq \{ \theta_1, \ldots, \theta_\ell \}$.  Without loss of generality, we may assume $k \geq \ell$ and $\lambda_1 \not\in \{ \theta_1, \ldots , \theta_\ell \}$.  Applying Lemmas \ref{lem:repRootPowerComp1} and  \ref{lem:repRootPowerComp3}, for each $1 \leq i \leq \ell$ there is $M_i \in \z_{\ge 1}$ such that $\langle g_i^{M_i} \rangle \cdot (g_i^{\bar s_i} ) w_{i}=0$ for all ${\bar {\bar s}}  \in \z_{\ge 0}^{m_0}$ such that $w_{\bar {\overline s}} \neq 0$.  Also, Lemma \ref{lem:degreehom} implies that $\langle f_i^{2n_i} \rangle v_i=0$ for each $1 \leq i \leq k$. Define $G=  t^L \prod_{i=2}^k f_i^{2n_i} \prod_{i=1}^\ell g_i^{M_i}$. Since the $\lambda_i$ are distinct and nonzero and $\lambda_1 \neq \theta_i$ for all $1 \leq i \leq \ell$, the polynomials $f_1^{2n_1}$ and $G$ are relatively prime.  Thus, there are $h_1, h_2 \in \cc[t]$ such that $h_1 f_1^{2n_1} + h_2 G=t^0=f_1^0$.
By the choice of $L$ and definition of $G$, 
\begin{align*}
 h_2 G \phi(v_0 \otimes v) &=\sum_{\bar{\bar s} \in \z_{\ge 0}^{m_0}} (h_2 G \cdot g^{\bar {\bar s}} w_0) \otimes w_{\bar {\bar s}} + g^{\bar {\bar s}}  w_0 \otimes (h_2 G \cdot g^{\bar {\bar s}}w_{\bar {\bar s}}) =0.
\end{align*}
On the other hand,
\begin{align}
h_2 G v_0 &= \left( (f_1^0- h_1 f_1^{2n_1}) v_1 \right) \otimes \cdots \otimes v \label{eqn:tensorisom7}\\
&\quad + \sum_{i=2}^k  v_1 \otimes \cdots \otimes h_2 G v_i \otimes \cdots  \otimes v \label{eqn:tensorisom8}\\
& \quad + v_1 \otimes \cdots \otimes h_2 G v \label{eqn:tensorisom9}. 
\end{align}
Since $\langle f^{2n_i}\rangle v_i=0$, it follows that (\ref{eqn:tensorisom8}) is zero. By the choice of $L$, (\ref{eqn:tensorisom9}) is also zero. Finally, (\ref{eqn:tensorisom7}) simplifies to $\left(f_1^0 v_1 \right) \otimes \cdots \otimes v$. 
Therefore, $h_2G v_0 \otimes v$ is a nonzero element in the kernel of $\phi$, which contradicts that $\phi$ is an isomorphism. Thus,  it must be that $k=l$ and (after possibly reordering) $\lambda_i= \theta_i$ for all $1 \leq i \leq k$.  In this case, $f_i=g_i$; we use the notation $f_i$ below. 

We next argue that $n_i= m_i$ and $\mu_i= \theta_i$ for all $1 \leq i \leq k$. Without loss of generality, we may assume $n_1  \leq m_1$. For each $1 \leq i \leq k$, choose $M_i \in \z_{\ge 1}$ such that $M_i \geq 2 n_i$ and $\langle f_i^{M_i} \rangle \cdot (f_i)^{\bar s_i} w_{i}=0$ for all ${\bar {\bar s}} \in \z_{\ge 0}^{m_0}$ such that $w_{\bar {\overline s}} \neq 0$.  Define $F=\prod_{i =2}^k f_i^{M_i} $.

Let $\bar {\bar s} = (\bar s_1, \ldots, \bar s_\ell) \in \z_{\ge 0}^{m_0}$ be such that $\bar s_1$ is maximal with $w_{\bar {\bar s}} \neq 0$. First suppose $| \bar s_1|>0$. If $\ell ( \bar s_1) > 0$, define $m= m_1 + k_1 + 1 - \ell ( \bar s_1)$; if $\ell ( \bar s_1 ) = 0$, define $m= m_1 + k_1  + s_{1, 0}$. Since $f_1^{M_1}$ and $F$ are relatively prime, there are $g_1, g_2 \in \cc [t]$ such that $f_1^m = h_1 f_1^{M_1} + h_2 F$.  Then, for any $j \geq L$, 
\begin{align}
(t^j h_2F-\theta_1(t^j f_1^m)) \phi(v_0 \otimes v) &= \sum_{\bar {\bar s}' \in \z_{\ge 0}^{m_0}} (t^j f_1^m - \theta_1(t^j f_1^m) )\cdot f_1^{{\bar s}_1'}w_1 \otimes \cdots \otimes w_{\bar {\bar s}' } \label{eqn:tensorisom1} \\
&\quad -\sum_{\bar {\bar s}' \in \z_{\ge 0}^{m_0}} t^jh_1f_1^{M_1} \cdot f_1^{{\bar s}_1'}w_1 \otimes \cdots \otimes w_{\bar {\bar s}' } \label{eqn:tensorisom1a} \\
& \quad + \sum_{\bar {\bar s}' \in \z_{\ge 0}^{m_0}} \sum_{i \geq 2}  \cdots \otimes t^jh_2F \cdot f_{i}^{{\bar s}_{i}'} w_i\otimes \cdots \otimes w_{\bar {\bar s}} \label{eqn:tensorisom1b} 
%\\&=\sum_{\bar {\bar s}' \in \z_{\ge 0}^{m_0}} \left((t^j f_1^m - \theta_1(t^j f_1^m)) f_1^{{\bar s}_1'}w_1 \right)\otimes \cdots\otimes w_{\bar {\bar s}} 
\end{align}
Note that (\ref{eqn:tensorisom1a}) and (\ref{eqn:tensorisom1b}) are zero by the choice of $M_i$. Applying Lemma \ref{lem:reducedegree} to (\ref{eqn:tensorisom1}), we may choose $j$ such that $0 \neq (t^jh_2F-\theta(t^j f_1^m)) \theta(v_0\otimes v)  \not\in \cc \theta(v_0\otimes v)$. 

 On the other hand, since $m_1 \ge n_1$, a similar calculation yields
 \begin{align}
 (t^j h_2F-\theta(t^j f_1^m)) (v_0 \otimes v) &= (\mu_1(t^j f_1^m) - \theta_1(t^j f_1^m)) v_0 \otimes v \in \cc v_0\otimes v. \label{eqn:tensorisom2}
 \end{align}
This forces $\phi(v_0\otimes v) \in \cc \phi (v_0\otimes v)$, a contradiction. Thus, it cannot be that $| \bar s_1|>0$.
 
 Now consider the case $\bar s_1 = (0, \ldots, 0)$. If we let $m=n_1$ and replace $\theta_1$ with $\mu_1$ in the calculations above, we get
 \begin{align}
(t^jg_2F-\mu_1(t^j f_1^{n_1})) \phi(v_0 \otimes v)
&=\sum_{\bar {\bar s}' \in \z_{\ge 0}^{m_0}} \left((t^j f_1^{n_1} - \mu_1(t^j f_1^{n_1})) w_1 \right)\otimes \cdots \otimes  w_{\bar {\bar s}'} \label{eqn:tensorisom3} \\
%&=-\mu_1(t^j f_1^{n_1}) \phi(v_0 \otimes v) \\
%&\quad + \sum_{\bar {\bar s}' \in \z_{\ge 0}^{m_0}}  \left(t^j f_1^{n_1} w_1 \right)\otimes \cdots \otimes  w_{\bar {\bar s}'} \\
(t^jg_2F-\mu_1(t^j f_1^{n_1})) (v_0 \otimes v) &= (\mu_1(t^j f_1^{n_1}) - \mu_1(t^j f_1^{n_1})) v_1 \otimes \cdots \otimes  v=0. \label{eqn:tensorisom4}
\end{align}
If $n_1<m_1$, then  $t^j f_1^{n_1} \not\in \langle f_1^{m_1} \rangle$ and thus $t^j f_1^{n_1} w_1 \not\in \cc w_1$. 
It follows from (\ref{eqn:tensorisom3}) that $(t^jg_2F-\mu_1(t^j f_1^{n_1})) \phi(v_0 \otimes v) \not\in \cc \phi (v_0 \otimes v)$. 
This contradicts (\ref{eqn:tensorisom4}). Thus, it must be that $n_1=m_1$.  In this case,  (\ref{eqn:tensorisom3}) becomes $(\theta_1(t^j f_1^{n_1})-\mu_1(t^j f_1^{n_1}) )\phi(v_0)$. Because $\mu_1$ and $\theta_1$ are determined by polynomials in $j$, this matches (\ref{eqn:tensorisom4}) for all $j \geq L$ if and only if $\theta_1= \mu_1$. 
By applying this argument to other terms in the tensor product and using the symmetry of the isomorphism relationship, we conclude $n_i= m_i$ and $\mu_i= \theta_i$ for all $1 \leq i \leq k$.

It remains to show that $V \cong W$. From the above arguments, for any $0 \neq v \in V$, there is some $w \in W$ such that $\phi (v_0 \otimes v) = w_0 \otimes w$. Define a linear map $\phi': V \rightarrow W$ by $\phi'(v)=w$.  We argue that $\phi'$ is an isomorphism.   

Fix $v \in V$, and choose $L \in \z_{\ge 0}$ such that $\Vir_{\geq L}v=\Vir_{\geq L} \phi'(v)=0$. Let $x \in U(\Vir)$. By Lemma \ref{lem:UUk}, there is $u \in U(\Vir_{\geq L})$ such that $uv=u\phi'(v)=0$, \ $xv_0=uv_0$ and $xw_0=uw_0$.  (Note that Lemma \ref{lem:UUk} applies to elements of $\Vir$. We may extend this to $U(\Vir)$ by inducting on the natural grading of $U(\Vir)$.)  Then,
\begin{align*}
\phi(xv_0 \otimes v) &= \phi(u v_0 \otimes v + v_0 \otimes u v ) \\
&= u \phi(v_0 \otimes v) = u(w_0 \otimes \phi'(v)) \\
&=uw_0 \otimes \phi'(v) + v_0 \otimes u \phi'(v) \\
&= xw_0 \otimes \phi'(v).
\end{align*}
%In other words, $\phi(f^{\bar {\bar s}} v_0 \otimes v) = f^{\bar {\bar s}} w_0 \otimes \phi'(v)$ for all ${\bar {\bar s}} \in \z_{\ge 0}^{n_0}$ and $v \in V$. 
In particular, since $\phi$ is a bijection, this shows that $\phi'$ is a bijection. To show that $\phi'$ is a homomorphism, we note
\begin{align*}
x \phi(v_0 \otimes v) &= x w_0 \otimes \phi'(v) + w_0 \otimes x \phi'(v)\\
\phi(x(v_0 \otimes v)) &= \phi(xv_0 \otimes v+ v_0 \otimes xv)\\
&= x w_0 \otimes \phi'(v) + w_0 \otimes \phi'(xv).
\end{align*}
Therefore, $w_0 \otimes x \phi'(v)=w_0 \otimes \phi'(xv)$ and $\phi'$ is a homomorphism.
\end{proof}

\section{More simple induced modules via tensor products} \label{sec:tensorapplications}
In this section, we show that the tensor products described in Theorem \ref{thm:tensorsimpleV} are isomorphic to new simple modules, induced from either polynomial algebras $\Vir^f$ where $f$ has multiple distinct roots, or from certain ``restricted polynomial" subalgebras.

\subsection{A general approach to induced modules from tensor products} \label{sec:GeneralTensor}
Before working with the specific induced modules described in this paper, we establish a general connection via tensor products between modules induced from different subalgebras.

  Suppose $\g$ is a Lie algebra over a field $\F$ with an infinite but countable basis. Let $\mathfrak a$ and $\mathfrak b$ be Lie subalgebras of $\g$ such that $\dim \mathfrak a, \dim \mathfrak b, \dim \mathfrak a \cap \mathfrak b = \infty$.  
 %(The results of the section hold for finite-dimensional Lie algebras, but the notation becomes more cumbersome to deal with both cases simultaneously. Therefore, we restrict to the case relevant to our application.) 
  
Suppose $\alpha : \mathfrak a \to \F$ and $\beta : \mathfrak b \to \F$ are homomorphisms, and regard $\alpha + \beta : \mathfrak a \cap \mathfrak b \to \F$.  Define a one-dimensional $\mathfrak a$-module $\F_\alpha$ by $x.1 = \alpha(x)$ for any $x \in \mathfrak a$ and an induced $\g$-module
$$
V_{\alpha} = U(\g) \otimes_{U(\mathfrak a)} \F_{\alpha}.
$$
Write $v_{\alpha} =1 \otimes 1 \in V_{\alpha}$, the canonical generator.
Similarly define $\F_{\beta}$, $V_{\beta}$, and $v_{\beta}$; and $\F_{\alpha+\beta}$, $V_{\alpha+\beta}$ and $v_{\alpha+\beta}$.

\begin{prop} \label{prop:generaltensor}
Suppose the module $V_{\alpha} \otimes V_\beta$ is cyclic  generated by $v_\alpha \otimes v_\beta$. Then $V_{\alpha+\beta} \cong V_{\alpha} \otimes V_\beta$.
\end{prop}

\begin{proof}
For any $x \in \mathfrak a \cap \mathfrak b$, 
$$
x (v_\alpha \otimes v_\beta) = (x v_\alpha) \otimes v_\beta+ v_\alpha \otimes (x v_\beta)=  (\alpha(x)+\beta(x)) v_\alpha \otimes v_\beta.
$$
Thus, by the definition of $V_{\alpha+\beta}$, there is a homomorphism $\phi: V_{\alpha+ \beta} \rightarrow V_\alpha \otimes V_\beta$ with $\phi(v_{\alpha+\beta}) = v_\alpha \otimes v_\beta$. Because $V_{\alpha} \otimes V_\beta$ is generated by $v_\alpha \otimes v_\beta$, it follows that $\phi$ is surjective.

It remains to show that $\phi$ is injective. To do this, we first establish some notation.  Fix pairwise disjoint, ordered sets $w = \{ w_1, w_2, \ldots \}$, \ $x = \{ x_1, x_2, \ldots \}$, \ $y = \{ y_1, y_2, \ldots \}$, and $z = \{ z_1, z_2, \ldots \}$ so that 
\begin{itemize}
\item $w$ is a basis for $\mathfrak a \cap \mathfrak b$;
\item $w \cup x$ is a basis for $\mathfrak a$;
\item $w \cup y$ is a basis for $\mathfrak b$; and 
\item $w \cup x \cup y \cup z$ is a basis for $\g$.
\end{itemize}
Note that any of these sets may be finite or empty. 

Let $\Omega$ denote the set of all (possibly empty) finite weakly increasing tuples of positive integers. For any ordered subset $q = \{ q_1, q_2, \ldots \}$  of $\g$ and ${\bf i} = ( i_1, \ldots , i_k ) \in \Omega$ define 
$$q_{\bf i} = \left\{ \begin{array}{ll} q_{i_1} q_{i_2} \cdots q_{i_k} & \mbox{if ${\bf i} \neq \emptyset$} \\ 1 & \mbox{if ${\bf i} = \emptyset$} \end{array} \right.$$
in the universal enveloping algebra $U( \g )$. If the set $\{ q_1, q_2, \ldots \}$ is finite or empty, we assume that $\Omega$ is suitably restricted so that $q_{\bf i} $ is defined. For ${\bf i} = (i_1, \ldots, i_k ) \in \Omega$, define $\ell ( {\bf i} )=k$.  
%Fix a basis $\{w_1, w_2, \ldots \} \cup \{x_1, x_2, \ldots \} \cup \{y_1, y_2, \ldots \} \cup \{z_1, z_2, \ldots \}$ for $\g$ so that 
%\begin{itemize}
%\item $\{ w_1, w_2, \ldots \}$ is a basis for $\mathfrak a \cap \mathfrak b$;
%\item $\{w_1, w_2, \ldots \} \cup \{x_1, x_2, \ldots \}$ is a basis for $\mathfrak a$;
%\item $\{w_1, w_2, \ldots \} \cup \{y_1, y_2, \ldots \}$ is a basis for $\mathfrak b$.
%\end{itemize}
%Note that any of these sets may be finite or empty. 

%Let $\Omega$ denote the set of all (possibly empty) finite weakly increasing tuples of positive integers. For any ordered subset $\{ q_1, q_2, \ldots \}$  of $\g$ and ${\bf i} = ( i_1, \ldots , i_k ) \in \Omega$ define 
%$$q_{\bf i} = \left\{ \begin{array}{ll} q_{i_1} q_{i_2} \cdots q_{i_k} & \mbox{if ${\bf i} \neq \emptyset$} \\ 1 & \mbox{if ${\bf i} = \emptyset$} \end{array} \right.$$

By the PBW Theorem and the construction of the corresponding modules, we have the following:
\begin{itemize}
\item $\{ z_{\bf i} y_{\bf j} v_\alpha \mid {\bf i}, {\bf j} \in \Omega \}$ is a basis for $V_\alpha$;
\item $\{ z_{\bf i} x_{\bf j} v_\beta \mid {\bf i}, {\bf j} \in \Omega \}$ is a basis for $V_\beta$; 
\item $\{ z_{\bf i} x_{\bf j} y_{\bf k} v_{\alpha+\beta} \mid {\bf i}, {\bf j}, {\bf k} \in \Omega \}$ is a basis for $V_{\alpha + \beta}$;
\item $\{ z_{\bf i} y_{\bf j} v_\alpha \otimes z_{\bf k} x_{\bf l} v_\beta \mid {\bf i}, {\bf j}, {\bf k}, {\bf l} \in \Omega \}$ is a basis for $V_\alpha \otimes V_\beta$.
\end{itemize}

Let ${\bf i}, {\bf j}, {\bf k} \in \Omega$. Then
\begin{align} \label{eqn:basistensor}
z_{\bf i} x_{\bf j} y_{\bf k} (v_\alpha \otimes v_\beta) %&= \sum (z_{\bf i'} x_{\bf j'} y_{\bf k'} v_\alpha) \otimes (z_{\bf i''} x_{\bf j''} y_{\bf k''} v_\beta) \\
&= (z_{\bf i} y_{\bf k} v_\alpha) \otimes (x_{\bf j} v_\beta) \\
&\quad + \sum_{{\bf i'} \neq {\bf i}} (z_{\bf i'} y_{\bf k} v_\alpha) \otimes (z_{\bf i''} x_{\bf j}  v_\beta) \nonumber \\
& \quad 
+ \sum_{{\bf j'} \neq \emptyset \ \mbox{or} \ {\bf k'} \neq {\bf k}} (z_{\bf i'} x_{\bf j'} y_{\bf k'} v_\alpha) \otimes (z_{\bf i''} x_{\bf j''} y_{\bf k''} v_\beta) \label{eqn:basistensor3}
\end{align}
where the sum is over ${\bf i'}, {\bf j'}, {\bf k'},\in \Omega$ such that $ {\bf i}' \subseteq {\bf i}$, ${\bf j}' \subseteq {\bf j}$, and  $ {\bf k}' \subseteq {\bf k}$; and ${\bf i''} = {\bf i} \setminus {\bf i'}$, ${\bf j''} = {\bf j} \setminus {\bf j'}$, and ${\bf k''} = {\bf k} \setminus {\bf k'}$.

Using standard straightening arguments, for ${\bf j'} \neq \emptyset$,
$$
z_{\bf i'} x_{\bf j'} y_{\bf k'} v_\alpha = %\alpha(x_{\bf j'}) z_{\bf i'} y_{\bf k'} v_\alpha + 
\sum_{{\bf q}, {\bf r} \in \Omega, \ \ell({\bf q}) + \ell({\bf r}) \leq \ell({\bf i'}) + \ell({\bf k'})} b({\bf q},{\bf r})  z_{\bf q} y_{\bf r} v_\alpha
$$
for some $b({\bf q},{\bf r}) \in \cc$. Applying the same argument to $z_{\bf i''} x_{\bf j''} y_{\bf k''} v_\beta$, the sum (\ref{eqn:basistensor3}) can be written as a linear combination of basis vectors $z_{\bf q} y_{\bf r} v_\alpha \otimes z_{\bf s} x_{\bf t} v_\beta$ such that
$$
\ell({\bf q})+\ell({\bf r})+\ell({\bf s})+\ell({\bf t}) \leq \ell({\bf i'}) + \ell({\bf k'}) + \ell({\bf i''}) + \ell({\bf j''}) 
< \ell({\bf i})+\ell({\bf j})+\ell({\bf k}).
$$

Now let $0 \neq v \in V_{\alpha+\beta}$. We may write 
$$v = \sum_{{\bf i}, {\bf j}, {\bf k} \in \Omega} c( {\bf i}, {\bf j}, {\bf k}) z_{\bf i} x_{\bf j} y_{\bf k} v_{\alpha+\beta},$$
 where $c( {\bf i}, {\bf j}, {\bf k}, {\bf l}) \in \cc$. 
Choose $( {\bf i}_0, {\bf j}_0, {\bf k}_0)$ so that $c( {\bf i}_0, {\bf j}_0, {\bf k}_0) \neq 0$ and $\ell ( {\bf i}_0) + \ell ({\bf j}_0) + \ell ({\bf k}_0) \ge \ell ( {\bf i}) + \ell ({\bf j}) + \ell ({\bf k})$ whenever ${\bf i}, {\bf j}, {\bf k}$ satisfy $c ( {\bf i}, {\bf j}, {\bf k} ) \neq 0$.  Then writing $\phi(v) = \sum_{{\bf i}, {\bf j}, {\bf k} \in \Omega} c( {\bf i}, {\bf j}, {\bf k}) z_{\bf i} x_{\bf j} y_{\bf k} (v_\alpha \otimes v_\beta)$ in terms of the standard basis, the equation (\ref{eqn:basistensor}-\ref{eqn:basistensor3}) implies that the coefficient of 
$( z_{{\bf i}_0} y_{{\bf j}_0} v_\alpha ) \otimes (x_{{\bf k}_0} v_\beta)$
is $c( {\bf i}_0, {\bf j}_0, {\bf k}_0) \neq 0$.  Thus,  $\phi(v) \neq 0$.
\end{proof}

We note that the proof of Proposition \ref{prop:generaltensor} shows that in this context, the map ${\rm Ind}_{\mathfrak a \cap \mathfrak b}^\g \to {\rm Ind}_{\mathfrak a}^\g \otimes {\rm Ind}_{\mathfrak b}^\g$ is injective regardless of whether ${\rm Ind}_{\mathfrak a}^\g \otimes {\rm Ind}_{\mathfrak b}^\g$ is cyclic.

\subsection{$V^f_\mu$ for an arbitrary polynomial $f$}
Now we apply the results of Section \ref{sec:GeneralTensor} to describe $V^f_\mu$ for an arbitrary polynomial $f$. In order to do this, we first establish two lemmas addressing the underlying homomorphisms $\mu: \Vir^f \rightarrow \cc$.

\begin{lem} \label{lem:polydeg}
Let $(\lambda, n, f; \mu, p, r)$ be as in Conventions \ref{conv}. Then for any $g \in \cc[t]$ such that $g(\lambda) \neq 0$, the homomorphism $\mu' = \mu \mid_{\Vir^{g f^n}}$ has the form $\mu'(t^j g f^n ) = q(j) \lambda^j$ for $j \in \z$, where $q$ is a polynomial such that $\deg (q) = \deg (p)$.
\end{lem}
\begin{proof}
Write $g = \sum_{i=0}^m c_i t^i$, where $c_m \neq 0$.  Then 
\begin{align*}
\mu ( t^j gf^n ) &=  \mu \left(  \sum_{i=0}^m c_i t^{i+j} f^n \right)\\
&=\sum_{i=0}^m c_i p(i+j) \lambda^{i+j} \\
&= \left( \sum_{i=0}^m c_i \lambda^i p(i+j)  \right) \lambda^j.
\end{align*}
Note that $q(j)=\sum_{i=0}^m c_i \lambda^i p(i+j)$ can be viewed as a polynomial in $j$. Moreover, if $\deg (p) =d$ and the degree $d$ coefficient of $p$ is $c$, then $q$ has leading coefficient $\sum_{i=0}^m c c_i \lambda^i = c g(\lambda) \neq 0$. This completes the proof.
\end{proof}

\begin{lem} \label{lem:polydeg2}
Let $(k ; \lambda_i, n_i, f_i ; \mu_i, p_i, r_i)$ be as in Conventions \ref{conv2}; and define $f= \prod_{i=1}^k f_i^{n_i}$ and $\mu=\sum_{i=1}^k \mu_i: \Vir^f\rightarrow \cc$.  Then $\mu$ has the form
$$
\mu(t^j f)=\sum_{i=1}^k q_i(j) \lambda_i^j
$$
for some polynomials $q_i$ such that $\deg(q_i)=\deg(p_i)$. 

Moreover, any homomorphism $\mu: \Vir^f \rightarrow \cc$ has such a decomposition: if $ \mu(t^j f)=\sum_i q_i(j) \lambda_i^j$ (where $\deg(q_i)<n_i$), then $\mu=\sum_{i=1}^k \mu_i$ where $\mu_i: \Vir^{f_i^{n_i}} \rightarrow \cc$ is associated to a polynomial $p_i$ such that $\deg (p_i) =\deg (q_i)$.
%if $\mu(t^j f)=\sum_i q_i(j) \lambda_i^j$ for some polynomials $q_i$, then $\mu= \sum_i \mu_i$ where $\mu_i: \Vir^{(t-\lambda_i)^{n_i}} \rightarrow \cc$ is a homomorphism of the form $\mu_i(t^j (t-\lambda_i)^{n_i}) = p_i(j) \lambda_i^j$ for polynomials $p_i$ such that $\deg(p_i) = \deg(q_i)$.
\end{lem}

\begin{proof}
For each $1 \leq i \leq k$, define $g_i = \prod_{i' \neq i} f_{i'}^{n_{i'}}$.  Then $f = f_i^{n_i} g_i$ and $g_i ( \lambda_i ) \neq 0$, so the first part of the claim follows immediately from Lemma \ref{lem:polydeg}.

Now consider the second part of the claim. For each $1 \leq i \leq k$ and $d \in \z_{\ge 0}$, define a homomorphism $\gamma_{i,d}:  \Vir^{f_i^{n_i}}  \rightarrow \cc$ by $\gamma_{i,d} ( t^j f_i^{n_i} ) = j^d \lambda_i^j$.  Applying Lemma \ref{lem:polydeg}, we know that the restriction of $\gamma_{i,d}$ to $\Vir^f$ satisfies $\gamma_{i,d}(t^j f) = q_{i,d}(j) \lambda_i^j$ for some polynomial $q_{i,d}$ with $\deg(q_{i,d})=d$.   In particular, for each $i$, the set $\{ q_{i,d} \mid d \in \z_{\ge 1} \}$ forms a polynomial basis. Therefore, for each $1 \leq i \leq k$, we may write $q_i = \sum_{ d \in \z_{\ge 1}} b_{i,d} q_{i, d}$ for some $b_{i,d} \in \cc$.   Define $\mu_i: \Vir^{f_i^{n_i}}  \rightarrow \cc$ by $\mu_i = \sum_d b_{i,d} \gamma_{i,d}$. Then for $1 \le i \le k$ and $j \in \z$,
$$\mu_i ( t^j f_i^{n_i} )  = \sum_d b_{i,d} \gamma_{i,d} (t^j f_i^{n_i} )= \sum_d b_{i,d} j^d \lambda_i^j = p_i (j) \lambda_i^j,$$
where $p_i (x) = \sum_{ d \in \z_{\ge 1}} b_{i,d} x^d$; and 
$$\mu_i (t^j f) = \sum_d b_{i,d} \gamma_{i,d}(t^j f) = \sum_d b_{i,d} q_{i,d}(j) \lambda_i^j = q_i(j) \lambda_i^j.$$
Thus, the $\mu_i$ satisfy the claim.
\end{proof}

For a homomorphism $\mu: \Vir^f \rightarrow \cc$, define the homomorphism to have {\it large degree} if $\deg(p_i)\geq n_i-2$ for each $1 \leq i \leq k$, in the decomposition in Lemma \ref{lem:polydeg2}. 

\begin{cor} \label{cor:genfsimple}
 Let $(k ; \lambda_i, n_i, f_i ; \mu_i, p_i, r_i)$ be as in Conventions \ref{conv2}; and define $f= \prod_{i=1}^k f_i^{n_i}$ and $\mu=\sum_{i=1}^k \mu_i: \Vir^f\rightarrow \cc$. Then 
$$
\bigotimes_{i=1}^k V_{\mu_i}^{(t-\lambda_i)^{n_i}} \cong V^{f}_{\mu}.
$$
Moreover, for any polynomial $f \in \cc[t] \subseteq \cc[t^{\pm}]$ with nonzero constant term and any homomorphism $\mu: \Vir^f \rightarrow \cc$, the module $V^{f}_{\mu}$ has such a tensor product decomposition; and $V^{f}_{\mu}$ is simple if and only if $\mu$ has large degree.
\end{cor}
\begin{proof}
The isomorphism claims follow by an inductive argument from Proposition \ref{prop:generaltensor} and Lemmas \ref{lem:cyclictensor} and \ref{lem:polydeg2}.
The simplicity claim follows from Theorem \ref{thm:tensorsimpleV} (taking $V$ to be the trivial module), Proposition \ref{prop:notsimplesmalldegree}, and Lemma \ref{lem:polydeg2}.
\end{proof}

\subsection{Modules induced from restricted polynomial subalgebras}
In this section, we fix a polynomial $f= \sum_{i=0}^p a_i t^i$, with $p \geq 0$ and $a_i \in \cc$ with $a_p=1$ and $a_0 \neq 0$.

For any $m \in \z_{\geq -1}$, define the following subalgebras of $\Vir$:
\begin{align*}
\B_m &= \{ z, e_j \mid j \geq m \}; \\
\B_m^f &= \{ z, x_j^f \mid j \geq m \},
\end{align*}
where $x_j^f = \sum_{i=0}^p a_i e_{j+i}$.  Note that $\B_m = \B_m^{t^0}$ and $\B_m^f= \B_m \cap \Vir^f$. We refer to the subalgebras $\B_m^f$ as {\it restricted polynomial subalgebras}. In this section, we describe modules of the form
$$
V^{f, m}_\mu = \Ind_{\B_m^f}^{\Vir} \cc_{\mu},
$$
where $\cc_{\mu}$ is a one-dimensional $\B_m^f$-module determined by a homomorphism $\mu: \B_m^f \rightarrow \cc$.  We show how to identify $V^{f, m}_\mu$ with a tensor product of the form $V_{\ddot \mu}^f \otimes \Ind_{\B_m}^{\Vir} \cc_{\hat \mu}$ for certain homomorphisms $\ddot \mu : \Vir^f \to \cc$ and $\hat \mu : \B_m \to \cc$.  Since the modules $ \Ind_{\B_m}^{\Vir} \cc_{\hat \mu}$ are cyclic and $\Vir^+$-locally annihilated, we can then apply Theorem \ref{thm:tensorsimpleV} to determine the simplicity of $V^{f, m}_\mu$.

We first study the homomorphisms $\mu: \B_m^f \rightarrow \cc$ in order to determine a ``decomposition" into $\ddot \mu$ and $\hat \mu$, using arguments similar to Lemma \ref{lem:mupoly}.  
A linear map $\mu: \B_m^f \rightarrow \cc$ is a homomorphism if and only if  
$$
0= \mu([x_j^f, x_k^f ]) = ( k-j) \left( \sum_{i=0}^p a_i \mu(x_{j+k+ i}^f) \right)
$$
for all $j, k \geq m$.
Writing $\mu_j = \mu(x_j^f)$, this is equivalent to the recurrence relation $\sum_{i=0}^p a_i \mu_{j+i} = 0$ for $j > 2m$.  In particular, for $m \leq j \leq 2m+p$, $\mu_j$ can take on any value; and $\mu_j$ for $j>2m+p$ are uniquely determined by $\mu_j$ for $2m+1 \leq j \leq 2m+p$.  As an alternative formulation, we  know (from \cite[Cor. 2.24]{Elaydi95}) that the recursive formula for $\mu_j$ has a unique solution of the form 
$$\mu_j = \mu(x_j^f) = \sum_{i=1}^k q_i(j) \lambda_i^j \quad \mbox{for} \quad j \ge 2m+1,$$
where the $\lambda_i$ correspond to writing $f = \prod_i ( t - \lambda_i)^{n_i}$ and the polynomials $q_i$ are determined by $\mu_{2m+1}, \ldots, \mu_{2m+p}$. %That is, $q_i(x) =\sum_{\ell=0}^{n_i-1} \alpha_{i, \ell} x^\ell$ where the coefficients $\alpha_{i, \ell}$ are the solutions to the system of linear (in the coefficients $\alpha_{i,j}$) equations
%\begin{equation}
%\sum_{i=1}^k \sum_{\ell=0}^{n_i-1} \alpha_{i, \ell} j^\ell \lambda_i^j = \mu_j, \quad  2m+1 \leq j \leq 2m+p. 
%\end{equation}

Now define a homomorphism $\ddot{\mu} : \Vir^f \rightarrow \cc$ by 
\begin{align*}
\ddot{\mu}_j:=\ddot{\mu}(x_j^f) &= \sum_{i=i}^r q_i(j) \lambda_i^j, \quad j \in \z;\\
\ddot{\mu}(z) &=0.
\end{align*}
Note that $\ddot \mu_j = \mu_j $ for $j \ge 2m+1$, but this is not generally true for $m \le j \le 2m$.  
Define another homomorphism $\hat{\mu} : \B_m \rightarrow \cc$ by 
$$\hat \mu (z) = \mu(z) \qquad \mbox{and} \qquad \hat \mu (e_j) = \hat \mu_j,$$
where $\hat \mu_j = 0$ for $j \ge 2m+1$; and for $m \leq j \leq 2m$, $ \hat{\mu}_j$ is determined by the system of linear equations
\begin{equation} \label{eq:muhat}
\sum_{i=0}^p a_i  \hat{\mu}_{j+i} = \mu_j-\ddot{\mu}_j, \quad m \leq j \leq 2m.
\end{equation}
Written in matrix form, this system is upper-triangular with $a_0\neq 0$ on the diagonal. Therefore, there is a unique solution. 
Also, it follows from the definition of $\hat \mu$ that $\hat \mu(x_j^f) = \mu(x_j^f) - \ddot \mu (x_j^f)$.  Therefore, $\mu=\ddot \mu + \hat \mu$
as homomorphisms from $\B_m^f=\Vir^f \cap \B_m$ to $\cc$.  Using this decomposition, we will say that $\mu$ has {\it large degree} if $\ddot \mu$ has large degree.

\begin{cor} \label{cor:restrictedtensor}
Fix $m \in \z_{\ge -1}$. Let $\ddot \mu: \Vir^f \rightarrow \cc$ and $\hat \mu: \B_m \rightarrow \cc$ be homomorphisms and define $\mu=\ddot \mu+ \hat \mu : \B_m^f \rightarrow \cc$. Then,
$$
V^f_{\ddot \mu} \otimes {\rm Ind}_{\B_m}^\Vir (\cc_{\hat \mu}) \cong \otimes V^{f,m}_\mu.
$$
Conversely, for any homomorphism $\mu: \B_m^f \rightarrow \cc$, $V^{f,m}_\mu$ has such a tensor product decomposition; and $V^{f,m}_\mu$ is simple if and only if $\mu$ has large degree and ${\rm Ind}_{\B_m}^\Vir (\cc_{\hat \mu})$ is simple.
\end{cor}

\begin{proof}
Note that the modules $\Ind_{\B_m}^{\Vir} \cc_{\hat \mu}$ are cyclic and $\Vir^+$-locally annihilated.  Then the isomorphism claims follow from Lemma \ref{lem:WVgen} and Proposition \ref{prop:generaltensor}.  The simplicity result follows from Theorem \ref{thm:tensorsimpleV} and Corollary \ref{cor:genfsimple}.
\end{proof}

\bigskip
The modules  ${\rm Ind}_{\B_m}^\Vir (\cc_{\hat \mu})$ have been well-studied. Following the literature, we introduce them here in three cases: $m=0$, $m=-1$, and $m>0$.  

Let $\theta: \B_0 \rightarrow \cc$ be a homomorphism. Define the associated {\it Verma module}
$$
M(\theta)= {\rm Ind}_{\B_0}^\Vir (\cc_\theta).
$$
  Since commutators must be sent to zero by $\theta$ and $[e_0, e_j]=je_j$, it follows that $\theta(e_j) = 0$  for $j \geq 1$. Writing $
\theta(e_0)=h$ and $\theta(z)=c$ (both of which may take on any value as $\theta$ varies), \cite{FF90} show that the module $M(\theta)$ is irreducible if and only if $h \neq h_{r,s}(c)$ for all $r, s \in \z_{\ge 1}$, where
\begin{equation} \label{eqn:KacDet}
h_{r,s} (c) =  \frac{1}{48} \left((13-c)(r^2+s^2) + \sqrt{(c-1)(c-25)}(r^2-s^2)-24rs-2+2c \right).
\end{equation}

Note that $h_{1,1}(c)=0$.  Therefore, $M(\theta)$ is reducible whenever $\theta(e_0)=0$. In this case, $M(\theta)$ has a submodule $N(\theta)$ so that the quotient can be realized as an induced module.
We construct the induced module as follows. Let $\xi: \B_{-1} \rightarrow \cc$ be a homomorphism, and define
$$
\overline M (\xi)= {\rm Ind}_{\B_{-1}}^\Vir (\cc_\xi).
$$ 
Again appealing to commutators, it follows that $\xi(e_j)=0$ for all $j \geq -1$. From \cite{FF90}, we have that $\overline M (\xi) \cong M(\theta) /N(\theta)$, where $\theta(e_j)=0$ for all $j \geq 0$ and $\theta(z)=\xi(z)$; and that $\overline{M}(\xi)$ is simple if and only if $\xi(z)\neq 1-6\frac{(p-q)^2}{pq}$ for any relatively prime integers $p, q \geq 2$.

Finally we define a Whittaker module as introduced in \cite{LGZ11}.  For $m \ge 1$, let $\psi: \B_{m} \rightarrow \cc$ be a nonzero homomorphism. In this case,  the commutator relations imply that $\psi(e_j)=0$ for $j \geq 2m+1$; and for $m \leq j \leq 2m$, $\psi(z), \psi(e_j) \in \cc$ may take on any value as $\psi$ varies.
Define 
$$L_{\psi} = {\rm Ind}_{\B_m}^{\Vir} (\cc_{\psi}).$$ 
\cite{LGZ11} show that $L_{\psi}$ is irreducible if and only if $\psi(e_{2m}) \neq 0$ or $\psi(e_{2m-1}) \neq 0$. 

We now show how the simplicity results for $M(\theta), \overline M(\xi), L_\psi$, along with the prior results of this paper, can be used to determine the simplicity of $V^{f,m}_\mu$ for any $f$, $m$, and $\mu$.  If we specialize to polynomials $f$ with no roots of multiplicity greater than $1$, then the modules $V^{f,m}_\mu$ are the modules studied in Section 5 of \cite{TZ16} and Section 6 of \cite{TZ13}.

\begin{cor}
Let $\mu: \B_{-1}^f \rightarrow \cc$ be a homomorphism. Then the induced module 
$$
V_\mu^{f, -1} \cong V_{\ddot{\mu}}^f \otimes {\overline M(\hat \mu)}
$$ 
is simple if and only if $\mu$ has large degree and $ \mu(z) \neq1-6\frac{(p-q)^2}{pq}$ for relatively prime integers $p, q \geq 2$.   
\end{cor}

\begin{proof}
This follows from Corollary \ref{cor:restrictedtensor} and \cite{FF90}. 
\end{proof}

\begin{cor}
Let $\mu: \B_0^f \rightarrow \cc$ be a homomorphism and write $\mu_j=\mu(x_j^f)$, $\mu(z)=c$.
Then 
$$
V_\mu^{f,0} \cong V_{\ddot{\mu}}^f \otimes M(\hat \mu).
$$ 
is simple if and only if 
$a_0^{-2} \sum_{i=0}^pa_i\mu_i \neq h_{r,s}(c)$ for any $r, s \in \z_{>0}$; and  $\mu$ has large degree.  \end{cor}

\begin{proof}
From \cite{FF90}, the Verma module $M(\hat \mu)$ is irreducible if and only if $\hat \mu_0 \neq h_{r,s}^c$ for any $r, s \in \z_{>0}$. On the other hand, for $m=0$, (\ref{eq:muhat}) becomes a single equation with solution $\hat{\mu}_0 = a_0^{-1} (\mu_0-\ddot{\mu}_0)$. Since $-a_0 \ddot \mu_0 = \sum_{i=1}^p a_i \ddot{\mu}_i$ and $\ddot \mu_i = \mu_i$ for $i>0$, it follows that $\hat{\mu}_0=a_0^{-2} \sum_{i=0}^pa_i\mu_i$.    

The simplicity result now follows from Corollary \ref{cor:restrictedtensor}. 

\end{proof}

\begin{cor}
Let $m \in \z_{\geq 1}$ and let $\mu: \B_m^f \rightarrow \cc$ be a homomorphism. Then 
$$
V_\mu^{f,m} \cong V_{\ddot{\mu}}^f \otimes L_{\hat{\mu}}.
$$
is simple if and only if 
$$\sum_{i=0}^p a_i \mu_{2m+i} \neq 0 \quad \mbox{or} \quad a_0 \sum_{i=0}^p a_i \mu_{2m+i-1} -2a_1 \sum_{i=0}^p a_i \mu_{2m+i} \neq 0;$$
 and $\mu$ has large degree.
\end{cor}

\begin{proof}
From \cite{LGZ11}, $L_{\hat{\mu}}$ is simple if and only if $\hat{\mu}_{2m-1} \neq 0$ or $\hat{\mu}_{2m} \neq 0$. From (\ref{eq:muhat}), we have $\hat{\mu}_{2m}= a_0^{-1} (\mu_{2m} - \ddot \mu_{2m})=a_0^{-2} \sum_{i=0}^pa_i\mu_{2m+i}$.  Similarly, 
\begin{align*}
\hat \mu_{2m-1}&=a_0^{-1} (\mu_{2m-1} -\ddot{\mu}_{2m-1} -a_1 \hat \mu_{2m})\\
&=a_0^{-2} (a_0 \mu_{2m-1} - a_1 \mu_{2m} - a_0 \ddot \mu_{2m-1} +a_1 \ddot \mu_{2m})\\
&=a_0^{-2}((a_0\mu_{2m-1} +a_1 \mu_{2m})-(a_0\ddot \mu_{2m-1}+a_1\ddot \mu_{2m}) -2a_1(\mu_{2m}-\ddot \mu_{2m}))\\
&= a_0^{-2} (\sum_{i=0}^p a_i \mu_{2m+i-1} -2a_1 a_0^{-1}\sum_{i=0}^p a_i \mu_{2m+i})
\end{align*}
The result now follows from Corollary \ref{cor:restrictedtensor}. 
\end{proof}

\end{document}